  \documentclass[11pt,a4paper]{article}   
\usepackage{pictex,color} 
 \usepackage{dcpic}
\usepackage{makeidx}
\makeindex            

 \usepackage{amsmath}  
 \usepackage{amssymb}  
 \usepackage{amsthm}   

 \title{\centerline     {  Harish-Chandra Modules over $\Bbb Z$}}

\author{ G\"unter Harder  } 
 
\newtheorem{satz}{Theorem}[section]
\newtheorem{prop}{Proposition}[section]
 
\begin{document}  
   
 \maketitle
 

 \tableofcontents 
 
\newcommand \iso{ \buildrel \sim \over\longrightarrow} 
\font\kapitel=cmbx10 scaled \magstep2
\newcommand\bal{\begin{align}}
\newcommand\eal{\end{align}}
\newcommand\bpm{\begin{pmatrix}}
\newcommand\epm{\end{pmatrix}}
\newcommand\bm{\begin{matrix}}
\newcommand\edm{\end{matrix}}
\newcommand\bc{\begin{cases}} 
\newcommand\ec{\end{cases}} 
\newcommand\beq{\begin{equation}}
\newcommand\eeq{\end{equation}}
 \newcommand \Imm{\text{   Im}}
\newcommand \ag{{\alpha}}
\newcommand \bg{{\beta}} 
\newcommand \g{ \gamma}
\newcommand\sa{s_{\ag}}
\newcommand\ssb{s_{\bg}}
\newcommand\saa{s_\ag\cdot\chi}
\newcommand\sbb{s_\bg\cdot\chi}
\newcommand \gag{{\gamma_\alpha}}
\newcommand \gbg{{\gamma_\beta}} 
\newcommand\gga{\gamma_\ag}
\newcommand\ggb{\gamma_\bg}
\newcommand\gzz{\V{\ul \gamma}^{\ul z}}
\newcommand\gzp{\V{\ul \gamma}^{\ul z^\prime}}
\newcommand\la{\lambda}
\newcommand\ta{\tilde{a}}
\newcommand\tb{\tilde{b}}
\newcommand\vv {{\sf v}} 
\newcommand\cp{{\chi^\prime}}
\newcommand\cd{{\chi^\dagger}}
\newcommand\ak{\alpha_{r_w(k)}}
\newcommand\tl{\tilde\lambda} 
\newcommand\tne{\tilde{n}_1} 
\newcommand\tnz{\tilde{n}_2} 
\newcommand\nee{n_1}
\newcommand\nz{n_2}
\newcommand\geins{\gamma_1}
\newcommand\gz{\gamma_2}
\newcommand\gMe{\gamma_1^{M_1}}
\newcommand\ech{{\epsilon(\mu)}}
\newcommand \ep{e(\phi)}
\newcommand \fgK{ \frak g,K_\infty}
\newcommand\GKZ{ \frak {g}_\Z,\KK}
\newcommand\GKe{ \frak {g}_\Z,\Ke}
\newcommand\tKK{\tilde{\KK}}
\newcommand \lgk {\Hom_{K_\infty}(\Lambda^\pkt(\fg/\fk),}
\newcommand \lgke {\Hom_{K_\infty}(\Lambda^1(\fg/\fk),}
\newcommand \lgkn {\Hom_{K_\infty}(\Lambda^0(\fg/\fk),}
\newcommand\lgt{\Hom_{T_c}( \Lambda^\pkt(\fg/\fk)}
\newcommand\Aql{A_{\frak q}(\lambda)}
\newcommand\Aqm{A_{\frak q}(\mu)}
\newcommand\Ts{T_{1,\text{split}}}
\newcommand\Tc{T_1^c}
\newcommand\mP{^\circ P}
\newcommand\rn{^\circ r}
\newcommand\wu{w_{\text{un}}}
\newcommand\cm{^\circ \frak{m}} 
\newcommand\cme{^\circ \frak{m}^{(1)}}
\newcommand\ck{^\circ \fk}
\newcommand\Mc{^\circ  M}
\newcommand\KMc{\KK^{^\circ  M}}
\newcommand\mm{\mu^{(1)}}
\newcommand\tm{{\tilde{\mu }}}
\newcommand\tmp{{\tilde{\mu }^\prime}}
\newcommand\mab{\mu_{\rm ab}}
\newcommand\pil{\pi_\infty^{(\lambda)}}
\newcommand \fg{ \frak g}
\newcommand \fm{ \frak m}
\newcommand \fa{ \frak a}
\newcommand \fu{ \frak u}
\newcommand\tfk{\tilde{\fk}}
\newcommand\fc{\frak c}
\newcommand\fU{\frak{U}}
\newcommand\fP{\frak{P}}
\newcommand\fQ{\frak{Q}}
\newcommand\slz{\frak{sl}_2}
\newcommand \red{\textcolor{red}{\text{red}}}
\newcommand \blue{\textcolor{blue}{\text{blue}}}
\newcommand \ft{ \frak t}
\newcommand \fb{ \frak b}
\newcommand \fp{ \frak p}
\newcommand \fk{ \frak k}
\newcommand \fnx{ \frak{n}(\ul x_f)}
\newcommand\fI{\frak{Ind}}
\newcommand\utl{u^{(\lambda)}_{t_p}}

\newcommand\me{\mu^{(1)}}

\newcommand \fnt{ \frak{n}(\ul t_p)}
\newcommand \FU{ {\frak U}(\fg)}
\newcommand \ZG{\frak Z(\fg)}
\newcommand \GZ{\Spec(\Z)}
\newcommand \GQ{\Spec(\Q)}
\newcommand \MZ{{\cal M}_\Z}
\newcommand \MZmc{{\mathcal {M}}_\Z}
\newcommand \MM{{\cal M}} 
\newcommand\wtM{ \widetilde{ \M\otimes_\Q\A_f}}
\newcommand\ch{\text{\bf ch}}
\newcommand\koker{\text{koker}}
\newcommand\ol{\otimes \Q_\ell}
\newcommand\fin{\text{fin}}
\newcommand\chx{\text{\bf ch}(\ul x_f)}
\newcommand\cht{\text{\bf ch}(\ul t_p)}
\newcommand\Gam{\Gamma}
\newcommand \into{\hookrightarrow}
\newcommand\SGK{{\cal S}^G_{K_f}}
\newcommand\BSC{ \overline{\SGK}}
\newcommand\PBSC{\partial\SGK}
\newcommand\SG{{\cal S}^G }
\newcommand\ST{{\cal S}^T }
\newcommand\Ge{G^{(1)} }
\newcommand\Te{T^{(1)}}
\newcommand\XT{X^*(T)}
\newcommand\XTe{X^*(T^{(1)})}
\newcommand\IC{I(F,\C)}
\newcommand\ip{\pi_f\circ\iota}
\newcommand\SMi{{\cal S}^{M_i} }
\newcommand\SMe{{\cal S}^{M_1}  }
\newcommand\SMz{{\cal S}^{M_2}  }
\newcommand\SM{{\cal S}^{M} }
\newcommand\SGBS{{\bar{\cal S}}^{G}}
\newcommand\BBS{\partial(\SGBS)}
\newcommand\SGKp{{\cal S}^G_{K^\prime_f}}
\newcommand\piKK{{ \pi_{K_f^\prime,K_f}}}
\newcommand\Ke{{\cal K}^{(1)}}
\newcommand\LGZ{{\frak g}_\Z}
\newcommand\LZK{\Lambda^\pkt(\fg_\Z/\fk_\Z)}
\newcommand\LGe{{\frak g}^{1)}_\Z}
\newcommand\KK{{\cal K}}
 \newcommand\SMKP{S^{M_P}_{K_f^{M_P}}}
 \newcommand\SMK{S^M_{K_f^M}}
  \newcommand\SMKQ{S^{M_Q}_{K_f^{M_Q}}}
\newcommand\HMK{ \HH^M_{K_f^M}}
\newcommand\CiM{\Coh(H^\pkt_!(\SGK,\tM_F))}
\newcommand\Ld{  L^2_{\text{disc}}(G(\Q)\bs G(\A_f)/K_f)}
\newcommand\Coh{\text{Coh}}
\newcommand\coh{\text{coh}}
\newcommand\cu{\text{cusp}}
\newcommand\arith{\text{arith}}
\newcommand\cusp{\text{cusp}}
\newcommand\trans{\text{trans}}
\newcommand\Lr{L^{\text{rat}}}
\newcommand\Lcoh{L^{\text{coh}}}
\newcommand\Lc{  L^2_{\text{cusp}}(G(\Q)\bs G(\A_f)/K_f)}
\newcommand\Wp{W_{\pi_\infty\times \ip}}
\newcommand\Wpc{W^{\text{cusp}}_{\pi_\infty\times \pi_f\circ\iota}}
\newcommand\Hsw{H_{\sigma_\infty}\otimes \M(w\cdot\lambda))}
\newcommand\Hs{H_{\sigma_\infty}}
\newcommand\Hsp{H_{\sigma_\infty^{\prime}}}
\newcommand\vgP{\vert\gamma_P\vert}
\newcommand\vd{\vert\delta\vert}
\newcommand\spitz[1]{\'{#1}}
\newcommand\schwer[1]{ \`{#1}}
\newcommand\BSB[1]{\partial_{#1} \overline{\SGK}}
\newcommand\sein{'s }
\newcommand \bs{\backslash}
\newcommand\V[1]{\vert{#1}\vert}
\newcommand\rs{\text {(resp. }}
\newcommand \Div{\hbox {\rm Div}}

\newcommand \type{\hbox {\rm type}}  
\newcommand \tr{\hbox{\rm tr}}
\newcommand \Tr{\hbox{\rm Tr}}
\newcommand \Gal{\hbox{\rm   Gal}}
\newcommand \Spur{\hbox{\rm  Spur}}
\newcommand \HupM{H^{\scriptscriptstyle\bullet} ({{{\frak u}_P},{\cal M}}) } 
\newcommand \teiltnicht{\mathrel{\raise1pt\hbox to
.5pt{$\scriptstyle/$\hss}\vert}} 
\newcommand \Sym{\hbox {\rm Sym}}
 \newcommand\Isa{I_{\sigma\otimes \vert \gamma_\alpha\vert^{z_\alpha}}} 
\newcommand\IsaT{I_{\Theta_P\sigma\otimes \vert \gamma_\alpha\vert^{-z_\alpha}}} 
 
 \newcommand \Ind{\hbox{\rm  Ind}}
  \newcommand \Iun{\hbox{\rm  Indunit}}
\newcommand \vol{\hbox{\rm   vol}}
\newcommand \os{\omega(\sigma)}
\newcommand \loc{ {\rm  loc}}
\newcommand \Eis{ {\rm  Eis}}
\newcommand\Seq{${\rm (Seq)}$ }
\newcommand \prolim{\displaystyle{\lim_{\leftarrow}}}
\newcommand \colim{\displaystyle{\lim_{n\rightarrow\infty}}}
\newcommand \Ens{{\bf Ens}}
\newcommand \Id{\hbox{\rm  \rm Id}}
\newcommand \rank{\hbox{\rm   rank}}
\newcommand \Ext{\hbox{\rm  \rm Ext}}
\newcommand \End{\hbox{\rm  \rm End}}
\newcommand \Top{{\bf Top}}
\newcommand \Vect{{\bf Vect}}
\newcommand \Mod{{\bf Mod}}
\newcommand \hol {\hbox{\rm   hol}}
\newcommand \ganz {\text{   int}}
\newcommand \Aut{{\hbox{\rm  \rm Aut}}}
\newcommand \cris{\hbox{\rm   cris}}
\newcommand \integ{ {\rm  int}}
\newcommand \Res{\hbox{\rm   Res}}
\newcommand \Hom{\hbox{\rm   Hom}}
\newcommand \ppfeil[1]{\buildrel #1\over \longrightarrow}
\newcommand \pkt{ \bullet}
\newcommand \Off{{\bf Off}}
\newcommand \tors{\hbox{\rm   tors}}
\newcommand \Ob{\hbox{\rm   Ob}}
\newcommand \munn{\mu_{p^m}^{\otimes (n+1)}}
\newcommand \mun {\mu_{p^m}^{\otimes n }}
\newcommand \munnd{\mu_{p^m}^{\otimes ( -n)}}
\newcommand \mund {\mu_{p^m}^{\otimes (- n -1)}}
 \newcommand \Gmk{\Gamma_{(m+k),\ell}}
 
\newcommand\Gl {{ \rm  Gl }}
\newcommand\Gsp{{\rm Gsp}}
\newcommand\Lie {{ \rm Lie }} 

\newcommand \ord{\hbox{\rm   ord}}
 \newcommand \Ten[1]{\otimes_{#1}}
 \newcommand \AG {{\Bbb G}_m\times\A^1\ }
\newcommand\zM[4]{\bpm #1 &#2 \cr#3& #4\epm } 
\newcommand \Hpkt{H^{\scriptscriptstyle\bullet}}
\newcommand\C{\Bbb C}
\newcommand\D{\Bbb D}
\newcommand\Dl{{\Bbb D}_\lambda}
\newcommand\F{\Bbb F}
\newcommand\K{\Bbb K}
\newcommand\MB{\Bbb M}
\newcommand\N{\Bbb N}
\newcommand\PP{\Bbb P}
\newcommand\R{\Bbb R}
\newcommand\Q{\Bbb Q}
\newcommand\BV{\Bbb V}
\newcommand\Z{{\Bbb Z}}
\newcommand\Zz{{\Bbb Z[\frac{1}{2}]}}
\newcommand\Zit{\Z[{\bf{i}},\frac{1}{2}]}
\newcommand\Zi{\Z[{\bf{i} }]}
\newcommand\bi{\bf{i} }
\newcommand\A{\Bbb A} 
\newcommand\BH{\Bbb H}
\newcommand\Gm{{\Bbb G}_m}
\newcommand\Ga{{\Bbb G}_a}
\newcommand\Sone{{\Bbb S}^1}
\newcommand\w{{\bf w}}
\newcommand\BF{\F}
\newcommand\BC{\C}
\newcommand\BP{\P}
\newcommand\Sh{{\cal S}}
\newcommand\Tp{T_{\frak P}} 
\newcommand\cO{{\cal O}}
\newcommand\HH{{\cal H}}
\newcommand\G{{\cal  G }}
\newcommand\B{{\cal  B }}
\newcommand\cA{{\cal A }}
\newcommand\cL{{\cal L }}
\newcommand\T{{\cal T }}
\newcommand\CF{{\cal F }}
\newcommand\CE{{\cal E }}
\newcommand\CL{{\cal L }}
\newcommand\CJ{{\cal J }}
\newcommand\M{{\cal M}}
\newcommand\cN{{\cal N}}
\newcommand\Ml{{\cal M_\lambda}}
\newcommand\MlZ{{\cal M}_{\lambda,\Z}}
\newcommand\MlZi{{\cal M}_{\lambda,\Z[i]} }
\newcommand\Mli{{{\cal M}_{\lambda,\Q{ (i) }}}}
\newcommand\MlC{{\cal M_{\lambda,\C}}}
\newcommand\Qi{{\Q(i)}}
\newcommand\Imi{I_{\mu,\Q(i)}}
\newcommand\tMl{{\tilde \Ml}}
\newcommand\cM{{\cal M}}
\newcommand\CC{{\cal C}}  
\newcommand\CD{{\cal D}}  
\newcommand\cD{{\cal D}} 
\newcommand\cDl{{\cal D}_\lambda} 
\newcommand\cDlZ{{\cal D}_{\lambda,\Z}}
\newcommand\CG{{\cal G}}
\newcommand\CO{{\cal O}}
\newcommand\CT{{\cal T}}
\newcommand\CU{{\cal U}}
\newcommand\GU {{ \frak U}}
\newcommand\Gu {{ \frak u}}
\newcommand\CB{{\cal B}}
\newcommand\X{{\cal X}}
\newcommand\J{{\cal J}}
\newcommand\NN {{ \cal N }}
\newcommand\cW{{\cal W}} 
\newcommand\cV{{\cal V}} 
\newcommand\cF{{\cal F}} 
\newcommand\BG{\text {BG}}
\newcommand\LZ [1]{\frak {#1}_\Z}
\newcommand\vt[1]{\vert #1\vert}
\newcommand\zT[2]{\bpm #1 & 0\cr 0 &#2\epm}
\newcommand\dT[3]{\bpm #1 & 0&0\cr 0 &#2&0\cr 0& 0&#3\epm}
\newcommand\dB[3]{\bpm #1 & *&*\cr 0 &#2&*\cr 0& 0&#3\epm}
\newcommand\Iu[1]{I_{#1}/\frak{u}I_{#1}}
\newcommand\ulm {\underline{m}   } 

\newcommand\gK{(\fg,K_\infty)}
\newcommand\Ku{K_{\infty}}
\newcommand\KuM[1]{K_{\infty}^{#1}}
\newcommand\Kuc{K_{\infty}^{(1)}}
\newcommand\Ttc{T_1\times_\R\C}
\newcommand\GAf{G(\A_f)/K_f}
\newcommand\LGAf{ L^2(G(\Q)\bs G(\A)/K_f)}
\newcommand\Ug{{\frak U}(\fg)}
\newcommand\Zg{{\frak Z}(\fg)}
\newcommand\up{\underline{p}}
\newcommand\um{\underline{m}}
\newcommand\ug{\underline{g}}

\newcommand\Sl {{ \rm  Sl }}
\newcommand\Sg {{{ \rm  Sp}_g }}
\newcommand\Sp {{{ \rm  Sp} }}
\newcommand\GSp {{{ \rm  GSp}}}
\newcommand\SO {{{ \rm  SO}}}
\newcommand\OO{{\rm O}}
\newcommand\SU {{{ \rm  SU}}}
\newcommand\GSR {{{\GSp_g(\R)}}}
\newcommand\SR {{{\Sp_g(\R)}}}
\newcommand\piM{\pi_M^{v,\text{odd}}}

\newcommand\GrS{{\rm Gr}_{\rm Siegel}}
\newcommand\GrK{{\rm Gr}_{\rm Klingen}}
\newcommand\dN{\buildrel\bullet\over\NN}
\newcommand\dBd{\buildrel\bullet\over B_d}
\newcommand\ul {\underline   } 
\newcommand\Spec{\hbox{\rm Spec}} 
\newcommand\tM{\tilde {\cal M}}
\newcommand\cX{{\cal X}}
\newcommand\cY{{\cal Y}}
\newcommand\cZ{{\frak Z}}
\newcommand\tcX{\tilde{\cal X}}
\newcommand\tcY{\tilde{\cal Y}}
\newcommand\tcZ{\tilde{\cal Z}}
\newcommand\tG{\tilde{G}}

\newcommand\hM   {\ul {\cal M}  }
\newcommand\tMZ{\tilde {\cal M}_\Z}
 \newcommand\SHK{S^H_{K_f^H}}
  \newcommand\SHM{S^M_{K_f^M}}
 \newcommand\Mx{ M_{1,\xi}}
    \newcommand\Mxq{{  M_{1,\xi}(\Q_p)}}
 \newcommand\Meq{{ M_1(\Q_p)}}
 \newcommand\IBG{ I_{B(\Q_p)}^{G(\Q_p)} }
\newcommand\tcZq{{ \tcZ}(\Q_p)}
 \newcommand\Ad{{\rm Ad}}
 \newcommand\ad{{\rm ad}}
  \newcommand\MG{  \vert \omega_G\vert }
    \newcommand\MH{  \vert \omega_H\vert }
      \newcommand\MT{  \vert \omega_T\vert }
        \newcommand\MGT{  \vert \omega_{T\bs G}\vert }
    \newcommand\tMGT    {\vert\tilde \omega_{{\cal H}\bs G}\vert}
        \newcommand\Gssr{G_{\rm ss,reg}}
\newcommand\reg{{\rm reg}}

\noindent

\noindent 
\section{Introduction}
Harish-Chandra modules play an important role in the theory of representations 
of semi-simple Lie-groups over $\R,$  in a certain sense they are the algebraic skeleton of 
a certain class of representations of semi simple real  Lie groups.

In this note we show that certain classes of Harish-Chandra modules have in a natural way 
a structure over $\Z.$   The Lie group is replaced by a split reductive group scheme  $G/\Z,$
its  Lie algebra is denoted by $\fg_\Z.$  On the group scheme $G/\Z$ we have a Cartan involution $\Theta$
which acts by $t\mapsto t^{-1}$ on  the split  maximal  torus and  the fixed point group scheme  $\KK/\Z$            of $\Theta$
 is a flat group  scheme over $\Z.$  A Harish-Chandra module over $\Z$
is  a $\Z$-module $\cV$ which comes with an action of the Lie algebra $\fg_\Z$, an action
of the group scheme $\KK$,  and we require some compatibility conditions between these two actions.
Finally we require that $\cV$ is a union of finitely generated  $\Z$ modules $\cV_I$    which are $\KK$
invariant.

The definitions are imitating the definition of a Harish-Chandra modules over  $\R$ or over $\C.$
(See for instance \cite{borel-wallach} 0.2.5, there these modules are called $(\fg ,K)$ modules.)

 For these $(\fg_\Z,\KK)$ modules $\cV$ we   define  cohomology modules $H^\pkt(\fg_\Z,\KK,\cV)$
 and these will be finitely generated $\Z$ modules provided the module $\cV$
 satisfies suitable finiteness conditions.
  We construct some simple examples, especially we construct the $\Z$- version of the
 discrete series representations of $\Gl_2(\R)$ and compute their cohomology. 
 
In the next section 
 we discuss  the process of induction: For a parabolic subgroup $P/\Z$ and a  $(\fm, \KK^M)$-
 module $\cV$ for its reductive quotient $M/\Z$ we define the induced module
  $ \fI_P^G \cV.$ 
  
  In the final section we study  intertwining operators  between some specific  induced Harish-Chandra modules
 $\fI_P^G \D_\mu, \fI_Q^G \D_{\mu^\prime}$ where $P,Q$ are maximal parabolic subgroups  of $\Gl_N/\Z.$
  Here we have to introduce some twisting, we achieve such a twisting by extending the scalars 
  from $\Z$ to  the function field $\Q(s) $ and define $ \fI_P^G \D_{\mu}\otimes s$  over $\Q(s).$ 
  Then our intertwining operators are defined  as integrals.
We can not  expect that they are defined over $\Q(s).$ But it turns out
  (and this is certainly not surprising) that they can be written down in terms of the form
  $\Gamma(s) R(s)$ with $R(s)\in \Q(s)$  and $\Gamma(s)$ is of course the $\Gamma$-function.  If the intertwining  operator is holomorphic 
  at $s=0$ we can evaluate   at $s=0$ and it turns out that our intertwining operator,
  which is defined by the transcendental process of integration, is essentially a power
  of $\pi$ times a non zero rational number (Theorem \ref{maintheorem}).
  
  This rationality result is used  in \cite{ha-ra} Thm. 7.48,  it can be formulated without
  reference  to rational integral structures on Harish-Chandra modules, we just have to choose
  the "right" basis in certain one dimensional vector spaces.
 
 The main reason why we develop  these concepts is an intriguing question  concerning 
 the  cohomology   of these modules and its behavior under the intertwining operators. 
 It turns out that the cohomology in certain situations is a free module of rank one 
over a small ring $R$ (for instance $\Z$, $\Zit,\dots .$). Then the intertwining operator 
 divided by the appropriate  power of $\pi$ induces an isomorphism between 
 these cohomology  modules  after  we tensor them by  the quotient field of $R.$  This isomorphism depends 
 on some data, for instance some highest weights. Our question is whether this isomorphism
 is already an isomorphism over the basic ring $R$ independently of the data. 
 
 This question has been investigated in \cite{ha-gldrei} in a special case and
 reduced to an combinatorial identity, which then was proved by 
 D. Zagier (see \cite{Za}) in an appendix to \cite{ha-gldrei}.
 This gives a positive answer to the question above in this  special case.
This is the only evidence I have that the question makes sense, except that it seems to be
a very natural one.  

In this note we work with certain specific choices of Cartan involutions. Such a choice provides
the so called maximal definite group schemes $\KK/\Z$, these group schemes are flat over $\Spec(\Z)$
and they are even reductive if we invert the prime 2. But we can also choose other 
maximal definite group schemes $\KK^\prime$ which are reductive at the prime 2 and perhaps non reductive 
at some other places. This suggests that we should  speak of sheaves of Harish-Chandra modules over $\Spec(\Z).$

  \section{  Harish-Chandra  modules over $\Z$ }\label{gkrat}
  \subsection{The general setup}\label{general}
 For any affine group scheme  $H/\Spec(\Z)$ we denote by $A(H)$  its algebra of regular functions.
 The affine algebra of the multiplicative group scheme  $\Gm$ is $A(\Gm)=\Z[x, x^{-1}] ,$ i.e. we choose the generator $\gamma_1$ of the character module
 $X^*(\Gm) ,$ it is given by the identity. Let $\Ga$ be the one dimensional additive group scheme the
 $A(\Ga)=\Spec(\Z[X]).$

  Let $G/\Spec(\Z) $ be a reductive connected group scheme, we assume that the derived group
  $G^{(1)}/\Spec(\Z)$ is a simply connected Chevalley scheme, the central torus  $C/\GZ$ should be split. 
  Let ${\frak g}_\Z, {\frak g}^{(1)}_\Z$ be the Lie algebras of $G/\GZ, \Ge/\GZ$ respectively, let $\fc_\Z$ be the Lie algebra of $C$. We have the split 
  maximal torus $T/\Spec(\Z),$  let
  $T^{(1)}/\GZ= T\cap G^{(1)}.$  We choose a Borel
  subgroup $B/\Spec(\Z) \supset T/\Spec(\Z).$ As usual we denote the character module 
  $\Hom(T,\Gm)$ by $X^*(T),$ we have the direct sum decomposition
  \begin{align}
X_\Q^*(T)= X^*(T)\otimes \Q = X_\Q^*(T^{(1)}) \oplus X_\Q^*(C),
\end{align}
we will always write $\gamma= \gamma^{(1)}+ \delta,$ this is the decomposition  of a character 
$\gamma\in X_\Q^*(T)$ into its semi simple and its abelian part.

 Let  $\Delta (\hbox{ resp. } \Delta^+\subset X^*(T)$)
  be the set of  roots (resp. positive roots), let $\pi =\{\alpha_1,\alpha_2,\dots, \alpha_r\} \subset \Delta$ be the set of simple positive roots. Let $\gamma_1,\gamma_2,\dots, \gamma_r\in X^*(T^{(1)})$ be the dominant fundamental weights, we extend them to elements in $X_\Q^*(T)$ by putting  the abelian part equal to zero.
The element $\rho\in X_\Q(T)$ is the half sum of positive roots.

  For any root $\alpha$ we have the root  subgroup scheme
  $U_\alpha/\GZ,$ we assume that for all simple roots  we have fixed an isomorphism
  
  $$\tau_\alpha : G_a/\GZ\iso U_\alpha/\GZ,$$
  
  i.e. we have selected a generator $e_\alpha$ of the abelian group $U_\alpha(\Z)\iso \Z.$
  
  \bigskip
  
  From our simple root $\alpha$ we also get a subgroup scheme $H_\alpha\subset G^{(1)}/\GZ$
  which is "generated" by $U_\alpha,U_{-\alpha}$  and which is isomorphic to $\Sl_2/\GZ.$ 
  It has a maximal torus $T_\alpha/\GZ\subset  T^{(1)}/\GZ$ which is the intersection  of the kernels
  of the fundamental weights $\gamma_\beta$ where $\beta\neq \alpha.$ 
  The choice of $\tau_\alpha$ is  the same as the choice of an isomorphism
 $$
 \tilde \tau_\alpha : \Sl_2/\GZ   \to H_\alpha
 $$
  which sends the diagonal torus to $T_\alpha$ and  on the $\Z$-valued points
 
 $$    \bpm  1 & 1 \cr 0 & 1\epm \to e_\alpha$$
 
 The derivative of $\tau_\ag$ defines a generator $E_\ag\in \Lie(U_\ag)$  Finally we define the coroot  $\alpha^\vee:  \Gm \iso T_\alpha $ which  is defined by the rule $<\alpha^\vee, \alpha>=2.$
 
Let  $\Theta$ be the unique automorphism of $\Ge/\GZ$ which induces $t\mapsto t^{-1}$ on $T^{(1)}$
and restricted to $H_\alpha$  and composed with  $\tilde\tau_\alpha^{-1}$ is the inner automorphism
given by the element 

$$ s_\alpha=\bpm 0 & 1 \cr -1 & 0 \epm.$$
 We call the pair $(G^{(1)},\Theta)$ an Arakelow Chevalley scheme.
 The automorphism restricted to $\Ge(\R)$ is of course a Cartan involution and the fixed point set
 $\Ge(\R)^\Theta= K_\infty^{(1)}$ is a maximal compact subgroup. Here we use this automorphism to
 give  the structure of a group scheme over $\Spec(\Z)$  to $K_\infty^{(1)}$. 
  To be more 
 precise: The group scheme of fixed points $\Ke/\GZ=({\Ge})^\Theta/\Spec(\Z)$ is a flat group scheme 
 over $\Spec(\Z)$, it is smooth and connected over $\Spec(\Zz).$  We    call $\Ke$ a  {\it  maximal  definite connected} subgroup scheme  of $\Ge/\Z.$
  We denote by $\fk_{\Zz} $  its Lie algebra over 
 $\Zz.$  
We put $\fk_\Z=\fg_\Z \cap \fk_{\Zz}. $ Here $\fk_\Z$ is a maximal sub algebra for which
the restriction of the Killing form is negative definite.    This   justifies the terminology.

If we have an extension of  the Cartan involution to $G/\Z$ then we can also look on the fixed point scheme
 $G^\Theta/\Z$ and define $\KK/\Z= G^\Theta .$ We are mostly interested in cases
 where this extension induces $ t\mapsto t^{-1}$ on $C/\Z$, then $\Ke $ is the connected component
 of the identity of $\KK/\Z.$ In general  we denote by $\KK $ a group scheme lying between $\Ke $ and $G^\Theta.$
 Then $\KK/\Ke$ is a finite  constant  group scheme which is isomorphic to $(\Z/2\Z)^s.$ 
 We also consider  larger  subschemes  
 of the form $\tKK=\Ke\cdot C^\prime,$ where $C^\prime $ is any subtorus of the split maximal torus $C.$ We call them {\it essentially maximal definite} subgroup schemes. They are also smooth over $\Zz$ and   the Lie algebra  is denoted by $ \tfk_{\Zz}.$  Again we define $\tfk_{\Z }=\tfk_{\Zz}\cap \fg_\Z.$
 
 For any ring $\Z\subset R$
we define the notion of a Harish-Chandra module  over $R$, or  equivalently a $(\LGZ, \KK)$- module 
over $R.$

\bigskip

1) This will  by a  projective   $R-$ module $\cV$ which  is the union of finitely generated projective  submodules   ${\cV_I},{I\in {\cal I}}$
such that $\cV/\cV_I$  is torsion free. 
We have an action of $\KK$ on $\cV$  which respects the $\cV_I.$

\bigskip
2) If $L$ is the quotient field of $R$ then every irreducible finite dimensional representation $\vartheta$
of $\KK\times L$ occurs with finite multiplicity in this module and we have the isotypical decomposition
$$  \cV\otimes L  =\bigoplus \cV(\vartheta) $$
where $\cV(\vartheta)$ is the $\vartheta$ isotypical component.

\bigskip
3) We have a Lie-algebra action of $\LGZ \otimes R$  on $\cV.$ 

\bigskip

4) The group scheme $\KK$ acts by the adjoint action on $\LGZ$  and the $R$ -module
homomorphism

 $$ ( \LGZ\otimes R) \otimes \cV \to \cV,$$
which is given by 3), is $\KK$ invariant.

\bigskip
5) The restriction of the 
Lie-algebra action of $\fg_\Z$ to the Lie-algebra $\fk_\Z =\Lie(\KK)$ is the differential
of the action of $\KK.$

\bigskip
Finally we formulate a finiteness condition

6) For any $I\in {\cal I} $ we find an $I_1\in {\cal I} $ such that $\cV_I\subset \cV_{I_1} $  such that the 
Lie algebra action of $\fg_\Z$ on $\cV$ induces an $R$ -bilinear map
\begin{align}
\fg_\Z \times \cV_I \to  \cV_{I_1}
\end{align}
 
We say that the  $(\fg_\Z,\KK)-$ module has a central character if   
 the Lie algebra of the center $\frak{c}_\Z=\Lie(C)$  acts by a linear map $z_V : \frak{c}_\Z\to  R.$

\bigskip
\subsection{Some comments}\label{comments}
This is almost the same as the usual definition of a Harish-Chandra  module except
that the field of scalars $\C$ has been replaced by $R$ and the 
action of the maximal compact group $K_\infty$  is replaced by the 
action of the group scheme $\KK.$

We want to remind the reader what it means that   the group scheme $\KK/\Spec(\Z)$
acts upon $\cV$ and $\cV_I .$ We recall that by definition $\KK/\GZ $ is a functor from the 
category of affine schemes $ Y\to \Spec(R)$ to the category  of groups. This means that for
any commutative ring $  R_1$  containing  $R$ we get an abstract group of $R_1-$ valued point
$G(R_1)$ which depends functorially on $R_1.$ Then the action of $\KK/\GZ$ on
the $R$ module $\cV$ provides for any $R_1$ an action of $\KK(R_1)$ on the $R_1$ module
$\cV\otimes_RR_1.$ We require that for all our finitely generated submodules
the module $\cV_I \otimes R_1$ is invariant under $\KK(R_1).$

In all examples which will be discussed below we take for $\cal I$  the set
of finite sets of isomorphism  classes  of irreducible representations of the group scheme $\KK.$
If $I=\{\vartheta_1,\dots, \vartheta_r\} $ then 
\begin{align}
\cV_I=  \cV \cap \oplus_{\nu=1}^r \cV(\vartheta_\nu)
\end{align}
In this case the requirement 6) is superfluous.

We call $\cV$ irreducible if $\cV\otimes L$ does not contain a proper $(\fg_\Z, \KK)$ 
submodule, we call it absolutely irreducible if    $\cV\otimes L_1$ stays 
  irreducible for any finite extension $L_1/L$.
  
  We saw already that we have some flexibility in the choice of $\KK.$ If we replace $\KK$ by the connected
  component of the  identity $\Ke$ then we can restrict the $(\fg_\Z,\KK)$ module to $(\fg_\Z, \Ke).$ It may 
  happen that the restriction of an irreducible module is not irreducible anymore. 

\subsection{ Motivation for this concept}\label{motivation}
This may look a little bit artificial. Let us   choose a dominant weight $\lambda\in X^*(T)$ and  construct  a highest weight module $\M_{\lambda,\Z}.$ This highest weight module
has a central character $\zeta_{\lambda}\in X^*(C).$   We are looking for  absolutely irreducible  Harish-Chandra modules
$\cV$ (over $\Z$ or a slightly larger ring) 
  having the central character   $z_V=- d\zeta_{\lambda},$  and  which have non trivial
   cohomology with coefficients in $M_{\lambda,\Z}.$
The cohomology is defined as the cohomology of the complex
\begin{align*}
\Hom_\KK( \Lambda^\pkt(\fg_\Z/\fk_\Z), \cV\otimes \M_{\lambda,\Z})
\end{align*}
where the definition of the complex is exactly the same as in the traditional situation (See for instance \cite{harder-book} Chap. 3, section 4).  Hence we define 
\begin{align}
H^\pkt (\fg_\Z, \KK, \cV\otimes \M_{\lambda,\Z})=H^\pkt(\Hom_\KK( \Lambda^\pkt(\fg_\Z/\fk_\Z), \cV\otimes \M_{\lambda,\Z}))
\end{align}

It easy to see that only the semi-simple component is relevant for the computation
of the cohomology, we have 
\begin{align}\label{CohGK}
H^\pkt (\fg_\Z, \KK, \cV\otimes \M_{\lambda,\Z}) = H^\pkt(\fg_\Z^{(1)},\Ke,\cV\otimes \M_{\lambda,\Z})^{\KK/\Ke} \otimes \Lambda^\pkt(\frak {c}_\Z)
\end{align}
We will see that that factor $\Lambda^\pkt(\frak{c}_\Z) $ is rather uninteresting. If we replace $\KK$ by a larger group
$\tKK=\Ke\cdot C^\prime$ then we define more generally 
\begin{align}
H^\pkt (\fg_{\Z}, \tKK, \cV\otimes \M_{\lambda,\Z})=H^\pkt(\Hom_{\Ke}( \Lambda^\pkt(\fg_\Z/\tfk_\Z ), \cV\otimes \M_{\lambda,\Z}))
\end{align}
(Observe the subscript at the $\Hom$  is $\Ke$ and not $\tKK$ as one might expect.) If we choose $C^\prime=C$
and  replace $\KK$ in (\ref{CohGK}) by $\tKK$ then the factor $\Lambda^\pkt(\frak {c}_\Z)$ is replaced by $\Lambda^0(\frak{c}_\Z)=\Z.$

\medskip
 We will be mainly concerned with the group scheme $G=\Gl_n/\Spec(\Z),$  the involution $\Theta$ will be the usual
 involution $g\mapsto ^t  g^{-1}.$     Our first aim will be to construct
 for a given highest weight module $\M_{\lambda,\Z} $  a  very specific absolutely irreducible $(\fg_\Z,\KK) $ module 
  $\D_\lambda$ 
 which has non trivial cohomology. More precisely:  The lowest degree were we find non trivial cohomology is
 $b_n=[\frac{n^2}{4}] $ (See \cite{ha-ra}, 3.1.5) and 
 
\begin{align}
H^{b_n}(   \GKe, \D_\lambda\otimes \M_{\lambda,\Z})\iso 
 \bc \Zz \omega_\lambda^+\oplus \Zz \omega_\lambda^- & n \text{ even }\cr
  \Zz \omega_\lambda & n \text { odd }
  \ec \otimes \Lambda^\pkt(\frak{c}_\Z) 
\end{align} 
 We still have the action of $\KK/\Ke=\Z/2\Z(=\pi_0(\Gl_n(\R))$  on the cohomology. This action is non trivial if
 $n$ is even  and the cohomology decomposes in a $+$ and  a $-$ eigenspace. (See \ref{sec:HGK}).
This will be relevant for the definition of the periods in \cite{ha-ra}.  

\medskip

If we take the tensor product $\D_{\lambda}\otimes \C$ then we get the usual Harish -Chandra modules over 
$\C$
which are   denoted by $\D_{\lambda}$ in  \cite{ha-ra}, 3.1 4.  We will call these  modules over
$\C$ the transcendental Harish-Chandra modules.  These special transcendental modules will be 
the only tempered modules which have cohomology and they contribute to the cuspidal cohomology (See \cite{ha-ra}, Sec. 5).

\section{First  examples}
\subsection{The case of the torus $\Gm$}
For the multiplicative group scheme $\Gm/\Z$ we have $\Lie(\Gm)_\Z=\Z H.$   We may choose for  the group scheme $\KK$ simply the subscheme   $\KK=\mu_2$ of second roots of unity. Then we can construct
a $ (\Z H, \KK)$ module $Z[\gamma\otimes m]$ for any pair $(\gamma, m)$ where $\gamma\in X^*(\Gm) $ and
where $m$ is an integer modulo two.  If $\gamma= x^n$ then the generator $H$  of $\Lie(\Gm)$  acts by multiplication by $n$
and the action of $\KK(\Z)$ is given by the sign character  $-1\mapsto (-1)^m.$ Therefore 
it is clear that these modules   $\Z[\gamma\otimes m]$ are the absolutely irreducible  $(\Lie(\Gm)_\Z,\KK)$
modules. The pairs $(\gamma,m)$ are called the characters of  Hecke type $-\gamma$, if $m=0$ then these are 
the rational characters. 
We can do essentially the same for any split torus $C,$  for any pair $\gamma\in X^*(C)$
and any $\epsilon: \KK =C(Z) \to \{\pm\}$ we can construct the $(\Lie(C)_\Z, \KK)$
module $\Z[\gamma\otimes \epsilon].$

\subsection{ The special case $\Gl_2/\Z$}\label{Glzwei}
We consider the special case $G=\Gl_2/\GZ$ with $\tilde\tau_\alpha=\Id.$ The group $\Gl_2(\R)$ 
has its discrete series  representations and the resulting   $({\frak g}_\R,\Ku)$ -modules. We want to show 
that these discrete series representations  are base extensions  of Harish-Chandra  modules over $\Z.$  

Inside 
$G$ we have the subgroup scheme 
$$ \tKK =  \{ \bpm a &  b \cr -b  & a \epm \} \subset \Gl_2.$$
The affine algebra of $\tKK$ is $A(\tKK)=\Z[a,b, 1/(a^2+b^2)].$
Let $\cO =  \Zi $ where $\bi^2=-1.$ We   define the flat group scheme $R_{\cO/\Z}(\Gm)$, its $R$ valued points
are $R_{\cO/\Z}(\Gm)(R)=(\cO\otimes_\Z R)^\times.$ We choose  an isomorphism
$$
  j: \tKK \iso  R_{\cO/\Z}(\Gm) 
   $$ 
which is defined by the rule

 $$
 j  : I= \bpm 0 & 1 \cr -1  & 0 \epm  \mapsto \bi  ,
 $$

 The group scheme $\tKK/\Z$ is not smooth, but  the embedding $\Zi\otimes \Zi \to \Zi\oplus \Zi $
 induces an embedding 
   
 \begin{align}
\tKK\times \Spec(\Zi)  \into  \Gm \times \Gm
\end{align}
 This embedding yields an inclusion  of affine algebras 
 $$
 \Zi[x,x^{-1}]\otimes \Zi[y,y^{-1}] \into A(\tKK)\otimes \Zi.
 $$ 
 Here  is $y= \bar{x}$  is the complex conjugate of $x.$ Then we get
 \begin{align}
a =\frac{1}{2} (x+y)  , b= \frac{1}{2\bi} (x-y)
\end{align}
This inclusion becomes an isomorphism if we invert $2.$
 We observe that we have the obvious inclusion $i: \Gm\into \tKK$  and we have the restriction 
 of the determinant $\det :\tKK \to \Gm.$  The kernel of $\det$ is the group scheme 
 \begin{align}
 \KK^{(1)} = \{ \bpm a &  b \cr -b  & a \epm \subset \Gl_2 \vert\;  a^2+b^2 =1\}
\end{align} 
 The character module $X^*(\KK^{(1)}\times \Zi) = \Z e, $ where 
\begin{align}\label{orient}
e:  \{ \bpm a &  b \cr -b  & a \epm \}   \mapsto (a+b\bi).
\end{align} 
   The matrix   
 \begin{align}\label{conj}
 c_2 = \bpm 1& 1\cr -\bf{i}&\bf{i}\epm \in \Gl_2(\Zit)
\end{align}
conjugates the standard diagonal torus $T\times \Zit$ into $\KK^{(1)}\cdot \Gm \times \Zit.$

\bigskip 
We  choose a weight $\lambda=l\g_1+d\delta$ where $ l\equiv 2d\mod 2.$ 
 We consider the space of regular functions on $G$ which satisfy
 \begin{align}
A_\lambda(B\backslash G)  =  \{f\in  A( G) \vert  f(bg)  =  \lambda(b) f(g)\} 
\end{align}
On this space of sections  we have the action of $G$ by right translations. The following is rather obvious and well
known
\begin{prop}\label{AG} The module $A_\lambda(B\backslash G)$ of regular functions 
  is trivial if $l>0. $ If $l\leq 0$ it realizes the module $\M_{\lambda^-,\Z}$  of highest weight
$\lambda^-=-l\gamma_1+ d\delta.$
\end{prop}

We can also say that $\lambda$ defines a line bundle $\cL_{\lambda}$ and 
\begin{align}\label{Borel-Weyl}
A_\lambda(B\backslash G)  = H^0(B\backslash G, \cL_\lambda).
\end{align}

The algebra of regular function  is   embedded into the  larger  algebra   $A(G)_e,$ 
these are the function which are regular at the identity element, it is the localization at $e.$
Again we define 
 \begin{align}
A_\lambda(B\backslash G)_e  =  \{f \in  A( G)_e \vert  f(bg)  =  \lambda (b) f(g)\}
\end{align}
 On this module we do not have an action
of $G,$ but it is clear that we still have  an action of $\fg_\Z.$ 

 We consider the morphism of schemes $m: B \times \tKK \to G $  given by the multiplication.   The intersection
  $B\cap \tKK= C =\Gm$ is embedded into the product   $t\mapsto (t,t^{-1}).$ The fiber of
the morphism $m$  are torsors under the action of $C.$ Then $m$ induces 
 a homomorphism of  affine algebras
  \begin{align}
A(G) \into ( A(B)\otimes A(\tKK)  )^C  \into A(G)_e
\end{align}
Our character $\lambda$ defines the rank one module $\Z[\lambda].$  Let $\lambda_C$ be the restriction of $\lambda$ to the center $C.$ Then the above embedding defines an inclusion
\begin{align}
\Z[\lambda]\otimes A_{\lambda_C}(\tKK) \into A(G)_e[\lambda]
\end{align}
The left hand  hand side is a $\tKK-$ module, this $\tKK$ module is invariant under the
action of the Lie-algebra $\fg_\Z. $  This allows us to define the induced module 
\begin{align}
\fI_B^G \Z[\lambda] = \Z[\lambda]\otimes A_{\lambda_C}(\tKK) 
\end{align}

For $\nu\equiv l \mod 2$ we define the elements 
\begin{align}
\Phi_{d,\nu} = \frac {(a+b\bi)^\nu}{(a^2+b^2)^{\frac{\nu-2d}{2}}} \in A_{\lambda_C}(\tKK)\otimes \Zi .
\end{align}
We get an inclusion
\begin{align}
 A_{\lambda_C}(\tKK)\otimes \Zi \supset \bigoplus \Zi \Phi_{d,\nu} 
\end{align}
and this inclusion becomes an isomorphism if we invert 2. Then we get a decomposition 
into eigenspaces under the action of $\KK.$ The $\Phi_{d,\nu}  $ are  characters. The complex conjugation ${\bf c}$ (the non trivial element in
$\Gal(\Q(\bi)/\Q)$) acts on the modules above and ${\bf c}(\Phi_{d,\nu} )=\Phi_{d,-\nu} .$

We define a submodule 
$$
 A^\prime_{\lambda_C}(\tKK) =(\bigoplus_{\nu\equiv l\mod 2} \Zi \Phi_{d,\nu}) \cap A_{\lambda_C}(\tKK) 
$$ if we invert 2 it becomes isomorphic to $ A_{\lambda_C}(\tKK).$

The Lie algebra $ \fg^{(1)}_\Z$ is a direct sum $\fg^{(1)}_\Z = \Z H \oplus \Z E_+\oplus \Z E_- ,$ where 
\begin{align} 
H = \begin{pmatrix} 1& 0\cr 0 &-1 \end{pmatrix} , E_+ = \begin{pmatrix} 0& 1\cr 0 &0 \end{pmatrix} ,
E_- = \begin{pmatrix} 0& 0\cr 1& 0 \end{pmatrix} .
\end{align}
We introduce some more notation
\begin{align}
\begin{matrix}
V= E_+ +E_- ,  & Y=  E_+-E_-\cr P_+ = H + i\otimes V,  &P_- = H -i\otimes V 
\end{matrix}
\end{align}
where the elements in the first row are in $  \fg^{(1)}_\Z$ the elements in the second row 
are in $\fg^{(1)}_{\Zi}.$
Under the adjoint action  of $\tKK$ the elements  $P_+, P_-$ are eigenvectors. We have 
\begin{align}
\Ad(k) P_+ = \Phi_{0,2}(k) P_+,  \Ad(k) P_- = \Phi_{0,-2}(k) P_-
\end{align}

  Some elementary computations yield 
 (the reader may find a more detailed exposition in  \cite{harder-book} in the  file [sl2neu.pdf]):
 \begin{align}\label{diff}
   Y \Phi_{d,\nu} = i\nu \Phi_{d,\nu} ,\; 
P_+ \Phi_{d,\nu} = (l + \nu ) \Phi_{d,\nu+2},\; 
P_- \Phi_{d,\nu} = (l -\nu ) \Phi_{d,\nu-2} 
\end{align}
We look at the   generators of $\fg_\Z/\fk_\Z$
 $$
H={1\over 2} (P_++P_-),  V=\frac {1}{ 2i} (P_+-P_-)
$$
and because of the parity conditions it is clear  
 $$ A^\prime_{\lambda_C}(\tKK)  $$
  is a $(\fg_\Z, \tKK)$- module and hence we get that
 \begin{align}
\fI_B^G \Z[\lambda]  \otimes \Zz \text{ is a  } (\fg_\Z, \tKK) \text{ module} 
\end{align}

 \bigskip
 Now it becomes clear that  $\fI_B^G \Z[\lambda]  \otimes \Zit $ is never irreducible. We have two cases.
 
  Let us first assume  $l\leq 0,$
 then we get from our formulas  (\ref{diff}) that $P_+ \Phi_{d,-l}=0, P_- \Phi_{d,l}=0$ and we find 
 a non trivial invariant submodule
      
 \begin{align}
\bigoplus_{l\leq \nu\leq -l:\nu\equiv l\mod 2} \Zit \Phi_{d,\nu}  
\end{align}
  and if we  look a little bit more closely then we see that this is the module $\M_{\lambda^-,\Zit}.$ The quotient by this submodule decomposes into a direct sum, i.e. we get an exact sequence
   \begin{align}\label{exact1}
0 \to \M_{\lambda^-, \Zit }    \to \fI_B^G \Z[\lambda]  \otimes \Zit  \to \cD_\lambda^+ \oplus \cD_\lambda^-\to 0
\end{align}
where
 
\begin{align}\label{discseries}
\cD_\lambda^+ =\bigoplus_{\nu\geq -l+2 ,\nu\equiv l\mod 2}\Zit \;\Phi_{d,\nu}\; ; \;
 \cD_\lambda^-=\bigoplus_{\nu\leq  l -2,\nu\equiv l\mod 2}\;\Zit \Phi_{d,\nu}
\end{align}
is a decomposition into two invariant submodules.

\bigskip 
We look at the second case where $l\geq 0.$ In this case  look at the induced module
$$\fI_B^G\Z[\lambda  +2\rho],$$
here $2\rho $ is the sum of the positive roots, in this case we have of course $2\rho=\alpha.$
In our formula (\ref{diff}) we have to replace $l$ by $l+2.$ We 
 have $P_-\Phi_{d,l+2}=0$ and $P_+\Phi_{d,-l-2}=0$
and hence we see that the two modules in (\ref{discseries}) are invariant submodules and we get an exact
sequence 
\begin{align}\label{exact2}
0 \to ( \cD_\lambda^+ \oplus \cD_\lambda^- )  \otimes \Zit  \to \fI_B^G \Z[\lambda+2\rho]  \otimes \Zit  \to \M_{\lambda }  \otimes \Zit   \to 0
\end{align}
For  any  $\lambda$
the modules $\cDl^{\pm}\otimes \C $ are the familiar  discrete series modules.   If we  consider the two weights   $\lambda=l\gamma_1+d\delta, \lambda^-=-l\gamma_1+d\delta$ and then  two discrete series  $\cDl^{\pm} , \cD_{\lambda^-}^{\pm} $ are not isomorphic
but if we take the tensor product with the rationals  then we find   isomorphisms

\begin{align}\label{isoint}
\Psi^{\pm}_{d,l}:     \cDl^{\pm}\otimes \Q  \to  \cD_{\lambda^-}^{\pm} \otimes \Q 
\end{align}
which is uniquely defined by the condition  $ \Psi^{\pm}_{d,l}(\Phi_{d,\pm(l+2)}) =\Phi_{d,\pm(l+2)}.$

\bigskip

In our notation the discrete series Harish-Chandra modules for $\Gl_2$ are  parametrized by
a pair $(\lambda,\text{sign})$ where $\lambda $ is a highest weight $\lambda=l\gamma_1+d\delta,l\geq 0.$
We have seen that  with these notation the exact sequences above tell us that
\begin{align}
\Ext^1(\M_{\lambda,\Zi},   \cDl^{\pm})\not=0.
\end{align}
 
 If we restrict the action of $\KK$ on $\cDl^{+} $ to $\KK^{(1) } $ then we get a decomposition into
 $\KK^{(1) }- $  types
 \begin{align}
\cDl^+  =\bigoplus_{\nu\geq l+2 \vert \nu\equiv l\mod 2} \Zi[\nu]
\end{align}
where $\KK^{(1) } $ acts by $ \nu e $ on $\Zi[\nu].$ The character $(l+2) e $ is called 
the {\it minimal $\KK^{(1) } $ type} in $\cDl^{+}.$ The character $-(l+2)e$ is also called
the  minimal $\KK^{(1) } $ type in $\cDl^{-}.$
  
In the following we will work with $\lambda=l\gamma_1+d\delta$ and $l\geq 0.$ We consider the 
module $\MlZi.$  Let us assume that we realized $\MlZ$ as the module of homogenous  
polynomials of degree $l$ 
in two variables $U,V. $ (This is actually the module $H^0(B\backslash G, \cL_{\lambda^-}).$) We consider the 
action of $\Ke\times \Zi $ on it and we have the decomposition into eigenspaces
\begin{align}
\Zit (U-i V)^l \oplus \dots   \oplus \Zit (U+iV)^l .
\end{align}
We are only interested in the highest and lowest weight vectors.  We abbreviate  $ (U-i V)^l =e_{-l},
(U+iV)^l =e_l.$ We also put $\cDl= \cD_\lambda^+ \oplus \cD_\lambda^-  .$  The relative Lie-algebra cohomology  with coefficients in $\cDl\otimes \MlZi $ is the cohomology  of the complex
\begin{align}
    \Hom_{\Ke}(\Lambda^\pkt(\fg_\Z/\tfk_\Z)\otimes \Zi, \cDl\otimes \MlZi)  
\end{align}
Here we observe that $\Lambda^0((\fg_\Z/\tfk_\Z)\otimes \Zi) = \Lambda^2((\fg_\Z/\tfk_\Z)\otimes \Zi)=\Zi$
where we choose $P_+\wedge P_-$ as generator.  Since $ \cDl\otimes \MlZi$ does not contain the
trivial $\Ke$ module this complex is  zero in degree 0 and 2.
In degree one we have
\begin{align}\label{PP}
\Lambda^1(\fg_\Z/\tfk_\Z)\otimes \Zi )=\fg_\Z/\tfk_\Z \otimes \Zi=  \Zi P_+\oplus \Zi P_-
\end{align}
We denote by $P_+^\vee, P_-^\vee\in \Hom((\fg_\Z/\tfk_\Z)\otimes \Zi,\Zi)$ the dual basis.  
 
\begin{prop} 
 $$H^1(\fg_\Z,\tKK , \cDl \otimes \MlZi ) =\Zi P_+^\vee\otimes \Phi_{d,l+2}\otimes e_{-l}\bigoplus \Zi P_-^\vee\otimes \Phi_{d,-l-2}\otimes e_{l}$$
\end{prop}
\begin{proof}
obvious
\end{proof}
If we replace $\tKK$ by $\KK=\Ke,$ then we get the same, but we have to multiply the right hand side 
by $\Lambda^\pkt{\frak c}_\Z.$

\medskip
It is clear from the construction,  that the element ${\bf c}$ in the Galois group acts on $\cDl$ and more precisely 
we  have ${\bf c}(P_+)=P_-,\; {\bf c}(\Phi_{d,\nu}) = \Phi_{d,-\nu},\;{\bf c}(e_l)=e_{-l}.$ 
Then we put
\begin{align}
\Omega_{d,l}= P_+^\vee\otimes \Phi_{d,l+2}\otimes e_{-l},\;\bar{ \Omega}_{d,l}= P_-^\vee\otimes \Phi_{d,-l-2}\otimes e_{l},
\end{align}
we may think of these elements as holomorphic and antiholomorphic 1-forms.

We   define $\cD_{\lambda,\Z}$
as  the $(\fg_\Z,\KK))$ module of  elements in $\cDl$ fixed by ${\bf c}.$
Then 
\begin{align}
\Hom_{\KK}(\Lambda^1(\fg_\Z/\tfk_\Z), \cD_{\lambda,\Z} \otimes \M_{\lambda,\Z})=\Z(\Omega_{d,l}+\bar{ \Omega}_{d,l})
\oplus \Z(i\Omega_{d,l}-i\bar{ \Omega}_{d,l})
\end{align}
We introduce some abbreviations
$$ \omega_{d,l}^{(1)}=\Omega_{d,l}+\bar{ \Omega}_{d,l},\; \omega_{d,l}^{(2)}=i\Omega_{d,l}-i\bar{ \Omega}_{d,l},
\;\eta=\bpm -1& 0\cr 0 & 1\epm 
$$
We still have the action of $\OO(2)/\SO(2)=\Z/2\Z=\pi_0(G(\R)).$     The nontrivial element is represented 
by the matrix $\eta $  defined above.  Under this action the  module 
$\Hom_{\Ke}(\Lambda^1(\fg_\Z/\tfk_\Z), \cD_{\lambda,\Z} \otimes \M_{\lambda,\Z})$ decomposes 
into a $+$ and a $-$ eigenspace. A straightforward computation shows that
\begin{align}
\eta(P^\vee_{\pm})=P^\vee_{-(\pm)}  ,\; \eta( \Phi_{d,\nu})= \Phi_{d,-\nu}, \eta(e_{\pm l})=(-1)^{\frac{2d-l}{2}}e_{-(\pm l)} 
\end{align}

\begin{prop}
The elements $\omega_{d,l}^{(1)},\;\omega_{d,l}^{(2)}\in \Hom_{\KK}(\Lambda^1(\fg_\Z/\tfk_\Z), \cD_{\lambda,\Z} \otimes \M_{\lambda,\Z})$ are generators  of the $\pm $ eigenspaces (maybe up to a power of  $2$). We have
$$ \eta(\omega_{d,l}^{(1)})=  (-1)^{\frac{2d-l}{2}}\omega_{d,l}^{(1)},\; \eta(\omega_{d,l}^{(2)})=-
(-1)^{\frac{2d-l}{2}}\omega_{d,l}^{(2)}$$
\end{prop}
\begin{proof}Again obvious
\end{proof}
We remember that $d \in \frac{1}{2}\Z$ and  satisfies $2d\equiv l \mod 2,$  hence it is well defined modulo $\Z.$
 Which of the two generators $\omega_{d,l}^{(1)},\;\omega_{d,l}^{(2)}$ is the generator of the $+$ eigenspace
 depends on $d$ and they change role if we replace $d$ by $d+1.$ This flip plays a decisive role
 in the definition of the periods in \cite{ha-ra}.
 
 Our module  $\cDlZ$ is irreducible but  its base extension $\cDlZ\otimes \Zi$ is reducible, it decomposes into
 $\cDl^+\oplus \cDl^-.$ If we enlarge $\Ke$  to $\KK = \Ke \rtimes <\eta>$ then  $\cDlZ\otimes \Zi$  becomes 
 an absolutely irreducible $(\fg_\Z, \KK )$-module.
 
 Then it is easy to see ( see for instance [sl2neu.pdf] that in the case $\lambda$ regular  (i.e. $l\not= 0,1$ )
 $\cDlZ$ is the only 
 irreducible $(\fg_\Z, \tKK)$ module which has non trivial cohomology  with coefficients in $\M_{\lambda,\Q}.$
 If $l=0$  then the trivial one dimensional $(fg_\Z,\Ke)$  -module $\Z$ has non trivial cohomology in degree
 $0$ and $2$ and this module completes the list of  modules which have non trivial cohomology  with
 coefficients in some $\M-{\lambda,\Z}.$
 
 \subsection{The intertwining operator}\label{IO}
 We come back to our highest  weight $\la =l\g_1+d\delta,$ we assume $l\geq 0.$ In equation (\ref{isoint})
 we wrote down  an intertwining isomorphism $\Psi_{d,l}$ between the two discrete series representation.
 If we look at the inverse of this operator and observe that $ \cD_{\lambda^-}^{+}\oplus  \cD_{\lambda^-}^{-}$
 is a quotient of $\fI_B^G\Z[\la^-]$ and   $ \cD_{\lambda }^{+}\oplus  \cD_{\lambda}^{-}$ is a submodule
 of  $\fI_B^G\Z[\la+2\rho]$ then our isomorphism provides  an intertwining operator 
 \begin{align}\label{Talgminus}
 T^{alg}_{\la^-}:\fI_B^G\Z[\la^-] \otimes \Q  \to \fI_B^G\Z[\la+2\rho]\otimes \Q 
\end{align}
which is  unique up to a scalar and we normalized  by fixing its value on a lowest $\KK$ type.

By the same token we get an operator in the opposite direction 
\begin{align}\label{Talgplus}
 T^{alg}_{\la+2\rho}:\fI_B^G\Z[\la+2\rho] \otimes \Q  \to \fI_B^G\Z[\la^-]\otimes \Q.
\end{align}
In this direction the space of homomorphisms is of rank one. The homomorphisms factor over a finite dimensional quotient.

In our situation here the maximal  torus  $T=\Gm\times\Gm$ and so far we only discussed the modules which are induced from rational characters. In this case we also may induce characters $\lambda\otimes\epsilon$ where 
$\epsilon: \tKK^T=\mu_2\times \mu_2\to \mu_2$ is a sign character, it is the form $(\pm 1,\pm 1)\mapsto (\pm 1)^{m_1} (\pm 1)^{m_2}.$ Then the induced module
$\fI_B^G \Z[\lambda\otimes \epsilon]$ is still reducible if the sign character $\epsilon=\ul m$ factors over the determinant, i.e.
we have $m_1=m_2.$ But if the sign character does not factor over the determinant then the induced 
module $\fI_B^G \Z[\lambda\otimes \epsilon] $ is in fact irreducible.

\bigskip
\subsection{Transcendental Harish-Chandra modules}\label{THC}
We return briefly to the transcendental theory of Harish-Chandra modules, we tensor everything by $\C$
and then our modules become Harish-Chandra modules in the traditional sense.  The group scheme $\Ke$
is replaced by  the group $\SO(2)=\Ku=\Ke(\R).$  The following is of course well known.

The  evaluation of the highest
weight $\la$  on $T(\R)$   provides an (algebraic)  character $\lambda_\R :T(\R)\to   \R^\times.$  We define
a larger class  of  (analytic) characters  $\chi:T(\R)\to \C^\times$ which are of the form  
\begin{align}
\chi =  (z, d,\ul m)  : \bpm  t_1 &  0\cr 0 &t_2\epm  \mapsto ( \vert \frac{t_1}{t_2} \vert)^{z/2}  \vert t_1t_2\vert ^{d} (\frac{t_1 }{\V{t_1}})^{m_1}( \frac{t_2}{\V{t_2}})^{m_2}
\end{align}
where  $z$ a complex variable and  $\ul m=(m_1,m_2)$ is a pair of integers  $\mod 2.$ The central
contribution given by the half integer $d$ should be fixed.  For us it seems to
be adequate to distinguish between the character $\la\in X^*(T)$ and its
evaluation $\la_\R.$  For $\lambda= l \gamma_1+d\delta$ we have
 $$
 \la_\R= \chi=(z,d,\ul m) \iff  z=l  \text{ and } m_1 \equiv \frac{l}{2}+d, m_2\equiv - \frac{l}{2}+d\mod 2
 $$ 
 We call such a $\chi$ algebraic, we say   that $\chi$ is cohomological  if $l\not=1.$ We say that $\chi$
 is of algebraic type if $z$ is an integer but the parity conditions may fail. 
 
We   define the induced representation
 \begin{align*}
I_B^{G,\infty} \chi = \{ f: \Gl_2(\R) \to \C\vert  f\in {\cal C}_\infty(\Gl_2(\R)),f( \bpm  t_1 &  u\cr 0 &t_2\epm )g =\chi(t)f(g)\},
\end{align*}
this is a $\Gl_2(\R)$ module. The submodule  of $K_\infty $ finite functions  is our  
   induced Harish-Chandra module $I_B^G \chi,$ for $\chi=\lambda_\R$ we have 
   \begin{align*}
I_B^G \lambda_\R  =\fI_G^G\lambda \otimes \C.
\end{align*} 
Let $m=m_1+m_2\mod 2$ then the 
  module is  a direct sum
\begin{align}
I_B^G \chi = \bigoplus_{\nu \equiv m\mod 2} \Phi^\chi_{\nu}
\end{align}
where 
\begin{align}
\Phi^\chi_{\nu}( \bpm  t_1 &  u\cr 0 &t_2\epm \cdot \bpm \cos(\phi) & \sin(\phi)\cr -\sin(\phi)& \cos(\phi)\epm )= (\vert \frac{t_1}{t_2} \vert)^{z/2}  \vert t_1t_2\vert ^{d} (\frac{t_1 }{\V{t_1}})^{m_1}( \frac{t_2}{\V{t_2}})^{m_2}  e^{2\pi i \nu\phi}
\end{align}
We have essentially the same formulae for the action of the Lie algebra
 \begin{align}\label{diffz}
   Y \Phi^\chi_{\nu} = i\nu  \Phi^\chi_{\nu} ,\; 
P_+  \Phi^\chi_{\nu} = (z + \nu )  \Phi^\chi_{\nu+2},\; 
P_- \Phi^\chi_{\nu}= (z -\nu ) \Phi^\chi_{\nu-2} 
\end{align}
note that the parity of $\nu$ is equal to the parity of $m_1+m_2.$ (see Slzweineu.pdf)

 For $\chi=(z,d,\ul m)$ we put $\chi^\prime=(-z,d,\ul m^\prime).$ Then we can write down 
the classical
(standard) intertwining operator
\begin{align}
T^{  st}_\chi    : I_B^G \chi   \to   I_B^G (\chi^{\prime}\otimes \rho^2)
\end{align}
which is defined by 
\begin{align}
T^{  st}_\chi  (f)(g)  =\int_{-\infty}^{\infty}  f (\bpm 0 & 1\cr -1 &0 \epm \cdot \bpm 1& u \cr 0 &1 \epm g )du
  \end{align}
where $du$ is of course the Lebesgue measure on $\R.$ This integral converges for $\Re(z)>>0$
and has an meromorphic  continuation into the entire $z$-plane. We need to locate the poles and we want to show
that this operator is never identically zero.

We introduce the notation $\chi^\dagger=\chi^{\prime}\otimes \rho^2.$
We evaluate it at the smallest $\Ku$ type, which is $\Phi^\chi_{ 0} $ if $m_1+m_2$ is even and  $\Phi^\chi_{ 1} $
if $m_1+m_2$ is odd. Let us put $\epsilon(\ul m)=0$ if $m_1+m_2$ is even and $\epsilon(\ul m)=1$ else.  Then an easy computation shows
\begin{align}
T^{  st}_\chi (\Phi^\chi_{ \epsilon(\ul m)})= \frac{\Gamma(\frac{z+\epsilon(\ul m)-1}{2})\Gamma(\frac{1}{2})} {\Gamma(\frac{z+\epsilon(\ul m)}{2})}  \Phi^\cd_{ \epsilon(\ul m)}
\end{align}
Then  we can evaluate $T^{  st}_\chi$ on any element $\Phi^\chi_{ \nu},$  we use the recursion provided by the formulae
(\ref{diffz}).   We get for $n\geq 1$
\begin{align}
P_+^n(\Phi^\chi_{ \epsilon(\ul m)})=(z+\epsilon(\ul m))\dots (z+\epsilon(\ul m)+ 2(n-1) ) \Phi^\chi_{ \epsilon(\ul m)+2n }
\end{align}
and on the other side (we have to replace $z$ by $2-z$) 
\begin{align}
P_+^n(\Phi^\cd_{ \epsilon(\ul m)})=(2-z+\epsilon(\ul m))\dots (2 -z+\epsilon(\ul m)+ 2(n-1)  ) \Phi^\cd_{ \epsilon(\ul m)+2n }
\end{align}
Therefore,  if $\nu=\epsilon(\ul m) +2n$
\begin{align}\label{Tst}
T^{st}_\chi(\Phi^\chi_{ \nu})=\frac{(2-z+\epsilon(\ul m))\dots ( 2-z+\epsilon(\ul m)+\nu-2)  )}{(z+\epsilon(\ul m))\dots (z+ \nu-2 )}   \frac{\Gamma(\frac{z+\epsilon(\ul m)-1}{2})\Gamma(\frac{1}{2})} {\Gamma(\frac{z+\epsilon(\ul m)}{2})} 
  \Phi^\cd_{ \nu}
\end{align}
(Note that here $\nu> 1,$  the product is empty if $\nu=0,1$ and hence it has value one if this is the case. Of
course we get a corresponding formula for $\nu\leq 0$).
We say that the intertwining operator  $T^{st}_\chi$ is holomorphic at $\chi=(z_0,d,\ul m)$ if for all
$\nu\equiv \epsilon(m)\mod 2 $  the function  $T^{st}_\chi(\Phi^\chi_{d,\nu})/\Phi^\cd_{d,\nu}$
is    holomorphic at $z_0.$   Otherwise we say that $T^{st}_\chi(\Phi_{d,\nu})$
has a pole at $z_0.$

\begin{prop} 
The intertwining operator $T^{st}_\chi$ has its poles  at the arguments $z_0=1-\epsilon(\ul m), -1-\epsilon(\ul m),\dots  $
and these are first order poles.
At these arguments  $T^{st}_\chi(\Phi^\chi_{ \nu})$ has a pole for all values $\nu.$  
\end{prop}
\begin{proof}
This is essentially an exercise in using the properties of the $\Gamma-$ function. We look at the denominator of 
the expression in (\ref{Tst}).  We have 
$$(z+\epsilon(\ul m))\dots (z+\nu-2 )\Gamma(\frac{z+\epsilon(\ul m)}{2}) =2^n\Gamma(\frac{z+\epsilon(\ul m)+2}{2}+n-1)$$ 
the $\Gamma-$ function has no zeroes,  hence the denominator does not contribute to poles. The $\Gamma-$ factor
in the numerator  has its poles exactly at the above list, these are first order poles and they do not cancel against 
the product of linear factors in front of the $\Gamma-$ factor.
\end{proof}
 
 We can form the composite $T^{  st}_{\cd} \circ T^{  st}_\chi $
and this is an endomorphism of $I_B^G \chi. $ Since for a general value of $z$ the module is irreducible
the operator must be a scalar $\Lambda(\chi)$  and it is not too difficult to write down this scalar.   We define $a(\ul m)=+1$ if $\epsilon(\ul m)=0$
otherwise $a(\ul m)=-1.$ 

 $$ \Lambda(\chi)= \frac{\Gamma(\frac{z-1+\epsilon(\ul m)}{2})
  \Gamma(\frac{1-z+\epsilon(\ul m)}{2})}  {\Gamma(\frac{z +\epsilon(\ul m)}{2})
  \Gamma(\frac{2-z+\epsilon(\ul m)}{2})} = 
 \frac{2}{z-1}\bigl(\frac{\sin(\frac{\pi}{2}z)} {\cos(\frac{\pi}{2}z)} \bigr)^{a(\ul m)} $$
 
 For us the important arguments for $\chi$ are the values $\chi=(l+2,m ) $ and $\chi=(-l,m ) $ where $l\geq 0$ is an integer
 and $l\equiv m\mod 2,$  we called these values of $\chi$ cohomological. Our proposition tells us that  $T^{  st}_\chi $ is holomorphic at 
 cohomological arguments.  But we also see
 that $\Lambda(\chi)$ vanishes at these arguments, i.e.  $T^{  st}_{\cp} \circ T^{  st}_\chi=0. $
Since it is clear that  the linear map $T^{  st}_\chi$ is never zero  it follows  that   $T^{  st}_\chi $ maps $I_B^G\chi$
to the kernel of $T^{  st}_{\cd}.$ 

This is of course consistent with our  results in section \ref{Glzwei},
if we tensorize the two exact sequences (\ref{exact1}),(\ref{exact2}) by the complex numbers and apply our intertwining operator
to the terms in the middle of the exact sequences then we get  for $\chi=\lambda_\R , \la=l\g_1+d\delta, l\geq 0$  
\begin{align}\label{Tplus}
T^{  st}_{\chi\otimes \rho_\R^2}   : I_B^G \chi\otimes \rho_\R^2 = \fI_B^G\Z[\la+2\rho]\otimes \C  \to \M_{\la,\C} \subset \fI_B^G\Z[\la^-]\otimes \C
\end{align}
and 
\begin{align}\label{Tminus}
T^{  st}_{\cd }: I_B^G \cd = \fI_B^G\Z[\la^-]\otimes \C \to ( \cD_\lambda^+ \oplus \cD_\lambda^- )  \otimes \C\subset
\fI_B^G\Z[\la+2\rho]\otimes \C.
\end{align}
These two intertwining operators are of course multiples of our earlier operators $ T^{alg}_{\la+2\rho}\otimes \C,
 T^{alg}_{\la^-}\otimes \C$. These earlier operators have been normalized
such that they gave the "identity" on certain $\Ku$ types. For the operator   $T^{alg}_{\la+2\rho}\otimes \C$
this is the $\Ku$-type $\Phi_{l}$ and for $T^{alg}_{\la^-}\otimes \C$ this is $\Phi_{d,l+2}.$ Then a straightforward
computation yields  for $\chi=\lambda_\R.$
\begin{align}
\begin{matrix}T^{  st}_{\chi\otimes \rho^2}  =\pi\;2^{\frac{3l-\epsilon(m)}{2}}(-1)^{\frac{l-\epsilon(m)}{2}} T^{alg}_{\lambda+2\rho}\cr
T^{  st}_{\cd}  =\pi \;\frac{1}{2^{\frac{l+2-\epsilon(m)}{2}}}(-1)^{\frac{l-\epsilon(m)}{2}} T^{alg}_{\lambda^-} 
\end{matrix}
\end{align}

This tells us that the operators $\frac{1}{\pi}T^{  st}_{\chi^{-1}},\frac{1}{\pi}T^{  st}_{\chi\otimes \rho^2}$ evaluated at cohomological arguments
are  defined over $\Q(i).$ They even induces n isomorphisms between the $\Zit$ modules of the 
cohomologically relevant  $\KK-$ types.  

We also have a brief look  at the induced modules which are not cohomological, these are the modules induced
from $\chi=(l,d,\ul m)$ where $l$ is an integer $2d\equiv l\mod 2, l-1\equiv \epsilon(\ul m)\equiv 0\mod 2.$ If now
$l-1 + \epsilon(\ul m)\geq 2 $ then  the operator $T_\chi^{st} $ is defined over $\Q(i),$ because
$\Gamma(1/2)$ appears in the numerator and in the denominator.  If $l+\epsilon(\ul m)-1=0,-2,-4,\dots$
then the intertwining operator has a pole. But we can modify the operator and define the normalized operator
\begin{align}
T^{norm}_\chi  = \frac{1}{\Gamma(\frac{z+1-\epsilon(\ul m)}{2})} T^{st}_\chi
\end{align}
This operator is holomorphic everywhere and at the arguments $\chi=(l,d,\ul m)$ which are not cohomological
it is an isomorphism and defined over $\Q(i).$

 \section{Induction of Harish -Chandra modules}
 \subsection{The general context}
  We pick a standard parabolic subgroup   subgroup $P/\Spec(\Z),$  let $U_P/\Spec(\Z)$ be its unipotent
  radical and $M=P/U_P$ its Levi quotient.  We can also view  $M/\Spec(\Z) $ to be  the Levi subgroup
  which is stable under the Cartan involution, this means that $M=P\cap P^\Theta.$    Then $\Theta$ induces 
  a Cartan involution on   the semi simple component $M^{(1)},$  it is simply
   connected. Let $ \KK^{M,1}\subset M^{(1)}$ be the fixed point scheme. Let $C_M$ be the connected center of $M$, it is a split torus.
   Then we define $\tKK^M= \KK^{M,1}\cdot C_M.$ The intersection $P\cap \Ke $ can be projected 
   down to $M$ and  yields a (possibly slightly larger) definite subscheme $\KK^M\supset \KK^{M,1}.$
   
  Let us assume
  we have a highest weight module $\M_{\mu,\Z} $ 
and a  $(\fm_\Z, \KK^M)$ Harish-Chandra module  $\cV$ over some ring $R$ for instance 
   $R=\Z, R=\Zi, R=\Zit.$     We give a construction of    the induced 
   $(\fg_\Z,\Ke)-$  module  $\fI_P^G \cV.$ We are interested in the case that $H^\pkt(\fm_\Z, \KK^M, \cV\otimes \M_{\mu,\Z}) \not=0.$ In this case we  compute the cohomology $H^\pkt(\fg_\Z,\Ke, \M_{\lambda,\Z})$   by adapting the method of Delorme.

   \bigskip
   We start from the module $\cV\otimes A(\KK). $ In the following $R_1$ will be a "variable" commutative ring
   containing $R.$  The algebra $A(\KK)$ is a $\KK\times \KK$ module,  recall that this means that for 
   an element $f\in A(\KK)\otimes R_1$  and $x, k_1,k_2\in \KK(R_1)$ we define 
  $$(R_{(k_1,k_2)} f)( k) = f(k_1^{-1}k k_2).$$ 
  Let $\KK^M= P\cap \KK. $ The projection $\KK^M\to  M $ is an injective homomorphism, we identify 
  $\KK^M$ with its image. This allows us to define the 
  submodule 
\begin{align}\label{defind}  
\begin{matrix}
(\cV\otimes A(\KK))^{\KK^M}=  \{\sum v_i\otimes f_i  \vert  \text { for all } k\in \KK^M(R_1) \cr \text{ we have } 
   \sum k v_i \otimes f_i = \sum v_i \otimes R_{(k^{-1} ,e)} f_i\} 
   \end{matrix} 
   \end{align}

     We show that this   is a $(\fg_\Z, \KK)-$ module!   The action of $\KK$ is is by translation from the right on the second factor:  For $k\in \KK(R_1) $ and any element $v\otimes f\in  (\cV\otimes A(\KK))^{\KK^M}\otimes R_1 $ we define 
     the translate 
      $$
      R_{(e,k)} (v\otimes f )= v\otimes R_{(e,k)} f.
 $$
   Now we have to define the $\fg_\Z$ action. We want to define 
   \begin{align}
\fg_\Z  \times  ( (\cV\otimes A(\KK))^{\KK^M}  )\to  ( (\cV\otimes A(\KK))^{\KK^M}  ).
\end{align} 
To do this we discuss again what happens on the $R_1$ valued points: For $X\in\fg_\Z $ and 
$\sum v_i\otimes f_i\in  (\cV\otimes A(\KK))^{\KK^M}\otimes R_1$ and $k\in \KK(R_1)$  we have to say what 
\begin{align}
X(\sum v_i\otimes f_i) (k )
\end{align}    
should be.  We work with dual numbers $R_1[\epsilon] $ then we should have 
$$
\epsilon X(\sum v_i\otimes f_i) (k )  =  (\sum v_i\otimes f_i) (k \cdot\text{exp}(\epsilon X)  ) - (\sum v_i\otimes f_i) (k)
$$
the pity is that the first summand on the right hand side is not yet defined. To define it  we consider 
the parabolic subgroup $k^{-1} P k \subset G\times R_1$  and observe that  the  linear map
\begin{align}
\Lie (k^{-1} P k ) \oplus \fk_\Z\otimes R_1  \to  \fg_\Z\otimes R_1
\end{align} 
is surjective. Hence we can write $X=  V + U $ where $V\in \Lie (k^{-1} P k ) , U\in \Lie(\KK)\otimes R_1.$
Now we can define 
\begin{align}
 (\sum v_i\otimes f_i) (k \cdot\text{exp}(\epsilon X)  )=  (\sum v_i\otimes f_i)( \text{exp} (\epsilon \Ad(k)(V) k\cdot  \text{exp} ( \epsilon U)) 
\end{align}
 the expression on the right hand side is defined. If we recall the definition of $(\cV\otimes A(\KK))^{\KK^M}$ then
 we see that it is equal to 
\begin{align}
\epsilon (\sum \Ad(k)(V) v_i \otimes f_i (k) +\sum v_i\otimes U f_i (k)) +\sum v_i\otimes f_i (k)
\end{align}
It it also clear that it does not depend on the  decomposition of $X=V+U.$

Hence we can define the induced Harish-Chandra module 
\begin{align}
\fI_P^G \cV  =  (\cV\otimes A(\KK))^{\KK^M}
\end{align}

It is not difficult to show that this satisfies all the conditions 1) to 6). Condition 2) may require
a longer argument. We will discuss an example in the following
section   and in this example it becomes clear why condition 2) is fulfilled.

\bigskip
   \subsection{The integral version of $\D_\lambda$}\label{Dl}
   We apply this induction process  to  a special case of the group $\Gl_n/\Z.$  We want to construct the $\Z$
structure on the modules which are called $\D_\lambda$ in \cite{ha-ra}, 3.1.4.  Let $T/\Z$ be the standard split torus
   and $B\supset T$ the standard Borel subgroup of upper triangular matrices.   The parabolic subgroups
   $P\supset B$ are the standard parabolic subgroups.
   
  Let $\g_1, \dots, \g_{n-1}\in X_\Q^*(T)$   be the dominant fundamental weights, let $\delta$ be the determinant.
For this we choose a self dual highest weight $\lambda= \sum_i^{n-1} a_i \gamma_i + d\delta,$ remember that self dual means $a_i=a_{n-i}.$ We use the usual construction to construct the $G/\Z$-module $\MlZ,$ it is the space 
of sections  $H^0(B\backslash G, \cL_{\lambda^-} )$ as in equation (\ref{Borel-Weyl}). We use the technique of induction to construct the very specific $(\fg_\Z , \Ke)$ modules $\Dl^\epsilon$  (where $\epsilon=\pm 1$)  over $\Zit$ which have  non trivial cohomology  in  lowest degree $b_n$
 \begin{align}
 H^{b_n}(\fg_\Z, \Ke, \Dl^\epsilon\otimes \MlZ)  \iso \Zit
\end{align}
(See also section \ref{motivation})

\medskip
We know that there is only a finite set of isomorphism classes  of irreducible Harish-Chandra  modules over $\C$ which have non trivial cohomology with coefficients in $\M_{\lambda}\otimes \C.$
  If $n$ is even (resp. $n$ is odd)  then there are only 
 two (resp. is only one)  $(\fg_\Z, \Ke )-$ module(s)  which are tempered or which can be the infinite component
 of a cuspidal representation. (See for instance \cite{ha-ra} 3.1 ,\cite{Speh} and \cite{Mog}).
  
 \subsection{The construction of $\Dl^\epsilon$}\label{Dle}
   We consider the parabolic  
  subgroup $^\circ P$ whose simple root system is described by the diagram
 
\begin{align}\label{circ} \circ  -\times -\circ -\times- \dots- \circ(-\times)
\end{align}
 i.e. the  set of simple roots    $\pi_{^\circ M}$ of the semi simple part of the Levi quotient $^\circ M$   consists of those 
simple  which have an odd index. This Levi subgroup can be identified  to
\begin{align}
\prod_{i:i \text{odd} } H_{\ag_i} = \prod \Gl_2 (\times \Gm)
\end{align}
 i.e. each factor is identified to $\Gl_2/\GZ,$  we have an extra  factor $\Gm$ if $n$ is odd. Let $m$
 be the largest odd integer less than $n.$ Note that here we have chosen a splitting of the
 Levi-quotient to a Levi subgroup, this splitting is unique, since we want that our Levi subgroup
 is stable under the Cartan involution.  Let $^\circ M^{(1)}$ be the semi simple component, we
 write  as usual $^\circ M =  {^\circ M}^{(1)}\cdot C_{^\circ M}. $
 
 The standard maximal torus is a product $T=\prod_{i:i \text{odd} } T_i( \times \Gm)$ 
 and for each $i=1,3,\dots, m $ we have 
 \begin{align}
X^*(T_i) \otimes \Q  =   \Q \gamma_{i}^{^\circ M^{(1)}}\oplus \Q \delta_i 
\end{align}  
 where $\gamma_{i}^{^\circ M^{(1)}}= \frac{\ag_i}{2} $ and $\delta_i$ is the determinant on that factor.
 For $n$ odd let $\delta_n$ be the character which sends the last entry $t_n$ to $t_n.$
 
 Let $\bar{B_i} \supset T_i$ be the standard Borel subgroup of upper triangular matrices 
 and let $\bar{B} = \prod_{i:i\text{odd}}{\bar B}_i$ be our Borel subgroup  of $\Mc.$ The  
 $\gamma_{i}^{^\circ M^{(1)}} $ are the dominant fundamental weights with respect to
 the choice of $\bar{B}.$

   We return to the conventions  in the first section and apply our considerations in section \ref{Glzwei} 
   to the factors $H_{\alpha_i}.$ The Cartan involution induces 
   the Cartan involution on each of the factors $H_{\ag_i},$  the group scheme $\prod_i \Ke_i=T_c$ is a maximal
   torus  in the reductive group  $\Ke.$ The character module $X^*(T_c\times \Zi)=\oplus_i \Zi e_i.  $ The Weyl  $W_c$ of this torus 
   acts on the character module by sending $e_i \mapsto \epsilon_i e_{\sigma(i)}$ where $\sigma$ is any permutations,
   where the $\epsilon_i=\pm 1$  and satisfy $\prod \epsilon_i =1$ if $n$ is even.
   
 Let  $B_c\supset \Ke \times \Zi$ be the  Borel subgroup  of 
   $ \Ke \times \Zi$ which contains $T_c$ and for which the   roots $e_1-e_{3},\dots, e_{m-2}-e_{m}, e_{m-2}+e_{m}$
   are the simple positive roots.

   \bigskip
   
 We have a very specific Kostant representative $w_{\rm un}\in W^{^\circ P}.$ The inverse of this permutation
 it is given by 
 
 $$ w_{\rm un}^{-1}=\{ 1\mapsto 1, 2\mapsto n, 3\mapsto 2, 4\mapsto n-1.\dots\}.$$
 The length of this element is equal to $1/2$ the number of roots in the unipotent
 radical of $\mP,$ i.e. 
 
\begin{align}\label{lowestdeg}l(w_{\rm un})=\bc  \frac{1}{4}n(n-2) & \text{  if  }  n \text{ is  even }\cr
 &\cr
\frac{1}{4} (n-1)^2  & \text{ if }   n \;  \text{ is odd } 
\ec
\end{align}

Then
\begin{align}
 w_{\rm un}(\lambda +\rho)-\rho = \sum_{i: i \hbox{ odd}} b_i\gamma_{i}^{^\circ M^{(1)}}   -(2\gamma_2+2\gamma_4+\dots+2\gamma_{m-1}  +\frac{3}{2}  \gamma_{m+1})
 + d\delta
\end{align}   
here $\gamma_2,\gamma_4, \dots $ are the dominant fundamental weights   which have an even index and the
$b_i$ are the cuspidal parameters
 $$  {b}_{2j-1}=\bc 2a_j+2a_{j+1}+\dots + 2a_{\frac{n}{ 2}-1} +a_{\frac{n }{2}}  +n  -2j &\text{  if } n \text{  is  even} \cr
 2a_j+2a_{j+1}+\dots + 2a_{\frac{n-1}{ 2} }+n   -2 j   & \text{ if }   n   \text{ is odd  }
 \ec
 $$ 
A simple computation shows that we can rewrite the expression for $ w_{\rm un}(\lambda +\rho)-\rho $

\begin{align}
\wu\cdot \lambda =\sum_{i: i \hbox{ odd}} ( b_i\gamma_{i}^{^\circ M^{(1)}}+  (c(i,n)+d)\delta_i )+ \begin{cases} 0 &\cr  
( - \frac{n-1}{2}+d)\delta_n  &  
 \end{cases}
 \end{align}
where the coefficients $c(i,n)$ are given by the formula

\begin{align}
c(i,n) = \begin{cases} \frac{n-i}{2} & \text{ if } n \text{ even }\cr
 \frac{n-1-i}{2} & \text{ if } n \text{ odd }
 \end{cases}
\end{align}

In this formula the summands $\mu_i= b_i\gamma_{i}^{^\circ M^{(1)}}+  (c(i,n)+d)\delta_i \in X^*(T_i)$ and 
$- \frac{n-1}{2}+d\in \Z.$   The sum of positive roots  in the $i$-th factor is $2\rho_i= \alpha_i=  2\gamma_{i}^{^\circ M^{(1)}}.$ We take the character $\mu_i+2\rho_i =( b_i+2)\gamma_{i}^{^\circ M^{(1)}}+  (c(i,n)+d)\delta_i ,$  and  apply the constructions from (\ref{Glzwei}) to it  and construct the module 
$\fI_{B_i}^{M_i}(\mu_i+2\rho_i). $  We know that this module sits in an exact sequence 
\begin{align}
0 \to \cD_{\mu_i} \to  \fI_{B_i}^{M_i}(\mu_i+2\rho_i)  \to  \M_{\mu_i,\Z} \to 0
\end{align}
We put $\mu=\wu\cdot \lambda ,$ this is a character on the maximal torus $T$ and we can define 
the induced module
$
\fI_{^\circ B}^{^\circ M} (\mu +2^\circ \rho). 
$
It is clear that this module is a tensor product 
\begin{align}
\fI_{^\circ B}^{^\circ M} (\mu +2{ ^\circ \rho}) =\bigotimes_{i: i\text{odd}}  \fI_{B_i}^{M_i}(\mu_i+2\rho_i) (\otimes \Z(- \frac{n-1}{2}+d))
\end{align}
where the last factor is only there if $n$ is odd. Then this module contains the submodule
\begin{align}
\cD_\mu=\bigotimes_{i: i\text{odd}}  \cD_{\mu_i}  (\otimes \Z(- \frac{n-1}{2}+d)) \into \bigotimes_{i: i\text{odd}}  \fI_{B_i}^{M_i}(\mu_i+2\rho_i) (\otimes \Z(- \frac{n-1}{2}+d))
\end{align}

  We know that   $\cD_{\mu_i}\otimes \Zit  $ decomposes into the two submodules 
 \begin{align}
  \cD_{\mu_i}\otimes \Zit  =\cD^+_{\mu_i}\otimes \Zit \oplus \cD^-_{\mu_i}\otimes \Zit ,
\end{align}
   hence for any choice of signs  we define the  module
   \begin{align*}
\cD_{\mu}^{\ul \epsilon}= \bigotimes_{i: i\text{odd}}  \cD^{\epsilon_i}_{\mu_i}  (\otimes \Z(- \frac{n-1}{2}+d)) 
\end{align*}
 and the induced module
   
     \begin{align}
      \Dl^{\ul{\epsilon}}=\fI_{^\circ P}^G \cD_{\mu}^{\ul \epsilon}.
\end{align}
  
   The  module $\   \cD_\mu^{\ul{\epsilon}}$ has as 
   minimal $\KMc$ type the character
  $$(\ul {\epsilon}, \mu+2\rho)= \sum_{i: i\text{odd}}  \epsilon_i(b_i+2)e_i   -(2\gamma_2+2\gamma_4+\dots+2\gamma_{m-1}  +\frac{3}{2}  \gamma_{m+1})
 + d\delta$$
    The    $\Zit$ eigenmodule  for this character is generated by 
    \begin{align}
 \cD_\mu^{\ul{\epsilon}}(\ul {\epsilon}, \mu+2\rho)=\Zit \bigotimes_{i:i\text{odd}} \Phi^{(i)}_{d,\epsilon_i(b_i+2)}
\end{align}
so it comes with a canonical generator, let us denote this generator by 
\begin{align}\label{gen}
\bigotimes_{i:i\text{odd}} \Phi^{(i)}_{d,\epsilon_i(b_i+2)}=\Phi_{\mu,\ul{\epsilon}}.
\end{align}

  The Weyl group $W_c$ contains a subgroup $S_m$ which acts by sign changes on the generators, i.e.
  $e_i\mapsto \pm e_i.$ Hence $S_m$ acts on the set of characters  $(\ul {\epsilon}, \mu). $  
   It acts transitively  on this set if $n$ is odd
    (See [Bou]) and if $n$ is even then we see easily that $(\ul {\epsilon}, \mu+2\rho) $ and   $(\ul {\epsilon^\prime}, \mu+2\rho)$
    are equivalent under the Weyl group $W_c$ if and only if $\prod_{i: i\text{odd}} \epsilon_i=\prod_{i: i\text{odd}} \epsilon_{i}^\prime.$
    
    Any of our characters $(\ul {\epsilon}, \mu+2\rho)$ can be conjugated by an element in $W_c$  into a dominant weight with respect to
    $B_c,$ and an easy computation shows  that these dominant weights are
    \begin{align}
\begin{cases}  ({\ul 1}, \mu+2\rho) = \sum_{i: i\text{odd}}  (b_i+2) e_i&  \text { if } n \text{ odd }\cr
 (\epsilon, \mu+2\rho) = \sum_{i: i\text{odd},i<m}   (b_i+2)e_i + \epsilon(b_m+2)e_m &  \text { if } n \text{ even }
 \end{cases}
\end{align}
where in the second case $\epsilon$ assumes the values $+1,-1.$ These weights are indeed dominant
because $b_{m-2}> b_m.$ For $\epsilon=\pm 1$ we define 
\begin{align}
\Dl^\epsilon = \Dl^{(1,1,\dots,\epsilon)}
\end{align}

\medskip
 We have the following
\begin{prop} 
The $(\fg_\Z, \Ke)$ modules $   \Dl^{\ul{\epsilon}}  $ are irreducible. Two such modules 
are isomorphic  if and only if
$(\ul {\epsilon}, \mu)$ and  $(\ul {\epsilon}^\prime, \mu)$ are conjugate under the Weyl group $W_c.$
The module $   \Dl^{ {\epsilon}}  $  contains a minimal $\Ke$ type which has highest weight 
$$\mu_c(\epsilon, \lambda)=  \sum_{i: i\text{odd},i<m}   (b_i+2)e_i + \epsilon(b_m+2)e_m   $$
where $\epsilon=1$ if $n$ is odd and   $\epsilon=\pm 1$ if $n$ is even. This  minimal $\Ke$
type occurs with multiplicity one.
\end{prop}
\begin{proof} 
For the irreducibility we tensor by $\C$ and  refer to \cite{Mog} and \cite{Speh}.
Any element in the  Weyl group $W_c$  can be represented by an element
$w\in G(\Z)$   which normalizes $T_c=\KK^{^\circ M}.$ Then 
 the multiplication from the left  by $w$ induces an isomorphism  
\begin{align}
 (  \cD_\mu^{\ul{\epsilon}}\otimes A(\Ke))^{\KK^{^\circ M}} \iso   ( \cD_{w\mu}^{w\ul{\epsilon}}\otimes A(\Ke))^{\KK^{^\circ M}}
\end{align}
and this is an isomorphism of $(\fg_\Z,\Ke)$ modules. We prove the assertion concerning the $\Ke$-types.
We have a decomposition of $\cD_\mu^{\epsilon} $ into $T_c$-types
\begin{align}\label{decoD}
 \cD_\mu^{ {\epsilon}}= \bigoplus_{k_1\geq0,\dots, k_m\geq 0} \Zit (( b_1 +2+2k_1)e_1 +\dots + \epsilon(b_m+2+2k_m)e_m)
\end{align}
The character $(b_1 +2+2k_1)e_1+ (b_3+2+ 2k_3 )e_3 +\dots + \epsilon(b_m+2 +2k_m)e_m $ may not 
be in the positive chamber ($k_m$ may be too large) but we can conjugate it to
$\mu_c(\ul k)  $ in the positive chamber. For this character it is easy to see that 
\begin{align}
\mu_c(\ul k)= 
(b_1 +2)  e_1+ ( b_3 +2)e_3 +\dots + \epsilon (b_m +2)e_m +\sum_i m_i\alpha_{i,c}  
\end{align}
where the $m_i\geq0.$ The highest weight $\mu_c(\epsilon, \lambda)=\mu_c(\ul 0).$

Now we have the classical formula that 

\begin{align}
 ( \Zit(\mu_c(\ul k))\otimes A(\Ke))^{\KK^M}\otimes \Q  =  \bigoplus_{\vartheta}  A(\Ke\otimes\Zit)(\mu_c(\ul k),\vartheta)\otimes \Q
\end{align}
where $\vartheta$ runs over the isomorphism classes of irreducible $\Ke\otimes\Zit$ modules and
where $A(\Ke\otimes\Zit)(\mu_c(\ul k))=\{f \vert R_{(t^{-1},e)}f =\mu_c(\ul k)) f \text{ for all } t\in T_c(R_1)\}.$ 
Then it is well known that the multiplicity of $\vartheta$ in $A(\Ke\otimes\Zit)(\mu_c(\ul k))$
is equal to the multiplicity of $\mu_c(\ul k)$ in $\vartheta.$  
We get that the  representation $\vartheta_{\mu_c(\epsilon, \lambda)} $ occurs with multiplicity one.
(We notice that our argument also implies that for a given $\vartheta,$ the number of those
$\ul k$ for which  $\mu_c(\ul k)$ occurs in $\vartheta$ with positive multiplicity, is finite.  The settles
the condition (2) in the definition of Harish-Chandra modules  in section \ref{general}
for $\fI_P^G \ \cD_\mu^{ {\epsilon}} $ but this argument works in the general case too.)

\end{proof}

   \subsection{The cohomology $H^\pkt(\fg_\Z,\Ke,  \Dl^\epsilon\otimes  \M_{\lambda,\Zit}) $} \label{sec:HGK}

   We define as usual the $(\fg_\Z,\Ke)$-cohomology  as the cohomology of the complex 
   
  \begin{align}\label{GKganz}
  \Hom_{\Ke}( \Lambda^\bullet(\fg_\Z/\fk_\Z),  \Dl^\epsilon \otimes \M_{\lambda,\Z})
  \end{align}
  
  If we tensor by the complex numbers then we know that $\Dl^\epsilon\otimes \C$ is unitary and  since 
  $\Ml\otimes \C$ is dual to its conjugate, it follows that all the differentials in the above complex 
  are trivial, i.e the complex is equal to its cohomology. 
  
  \bigskip
   We apply the Delorme method (or Frobenius reciprocity).  Let $\cm_\Z$ be the Lie algebra 
   of $^\circ M ,$ let  $ \cme _\Z$  be the Lie algebra of $^\circ M^{(1)}.$  Let $\fu_\Z$ be the Lie-algebra 
   of the unipotent radical of $^\circ P$  and finally let $\fc_\Z$
   be the Lie algebra of $C_{^\circ M}.$ Then
   \begin{align}
   \fg_\Zz/\fk_\Zz  =  \cme_\Zz/ \ck_\Zz\oplus \fc_\Zz\oplus \fu_\Zz 
   \end{align}
   where now the right hand side is a $(^\circ \fm, \KK^{^\circ M})$ -module. The group scheme
   $ \KK^{^\circ M} $  acts by the adjoint action. It acts trivially on $\fc_\Z,$  and the adjoint action  of $ \KK^{^\circ M} $ on $\fu_\Z$
   extends to the adjoint action of $^\circ M.$ (Remember that $\Mc$ is a subgroup of $^\circ P$.)
 We get an isomorphism of complexes 
   \begin{align}\label{Del1}
     \Hom_{\Ke}( \Lambda^\bullet(\fg_\Z/\fk_\Z),  \Dl^\epsilon \otimes \M_{\lambda,\Z})  = \\ \Hom_{\KK^{\Mc}} (\Lambda^\bullet( \cme_\Z/ \ck_\Z) ,  \cD_\mu^\epsilon  \otimes
   \Hom(  \Lambda^\bullet(\fu_\Z) , \M_{\lambda,\Z})\otimes \Lambda^\pkt(\fc_\Z)
   \end{align}
   
In the following we concentrate on the $\Lambda^0(\fc_\Z)$ component.   We claim that the $\Zit$ module
   \begin{align}\label{Hu}
 \Hom_{\KK^{\Mc}} (\Lambda^\bullet( \cme_{\Zit}/ \ck_{\Zit}) ,  \cD_\mu^\epsilon  \otimes
   \Hom(  \Lambda^\bullet(\fu_{\Zit}) , \M_{\lambda,\Zit}))
\end{align}
is free of rank one.  We will  be more precise: We decompose the three $\KK^{\Mc}$ -modules 
$\Lambda^\bullet( \cme_{\Zit}/ \ck_{\Zit}),  \cD_\mu^\epsilon , \text{ and }
   \Hom(  \Lambda^\bullet(\fu_{\Zit}) , \M_{\lambda,\Zit})$ into eigenspaces with respect to characters in $  X^*(\KK^{\Mc}\otimes \Zit)$
   and show that there is exactly one  triple of characters which contributes non trivially to the $\Hom_{\KMc}$, i.e
   which satisfies $\eta_{\fm}=\eta_{\cD}+\eta_{\fu}.$

\bigskip
  The module $\Lambda^{^\circ r}(\cme_{\Zit}/ \ck_{\Zit})$
contains the submodule
\begin{align}
\bigoplus_{\ul \epsilon} \Zit P_1^{\epsilon_1}\wedge P_3^{ \epsilon_3}\wedge\dots \wedge P_m^{\epsilon_m}
\end{align}
and on the individual summand our torus $T_c\cdot C_{\Mc}$     acts by characters
$
\nu(\ul \epsilon )= 2 (\epsilon_1 e_1+\dots+ \epsilon_m e_m).C_{\Mc}
$ 
We choose for $\ul \epsilon$ the value $\ul \epsilon_0=(+,+,\dots,\epsilon) $ hence $\Lambda^{^\circ r}(\cme_{\Zit}/ \ck_{\Zit})$
contains the direct summand 
\begin{align}
\Lambda^{^\circ r}(\cme_{\Zit}/ \ck_{\Zit})(\nu(\ul \epsilon_0))=\Zit(\nu(\ul \epsilon_0))
\end{align}
The character $\nu(\ul{\epsilon}_0)$ will be our $\eta_\fm.$  The action of $C_{^\circ M}$   on $\Lambda^{^\circ r}(\cme_{\Zit}/ \ck_{\Zit})$  is trivial.

\bigskip
The module $ \cD_\mu^\epsilon $ contains the submodule $ \cD_\mu^\epsilon(\mu+2\rho,\ul{\epsilon}_0 )$
with multiplicity one, hence 
\begin{align}
 \cD_\mu^\epsilon(\mu+2\rho,\ul{\epsilon}_0 )=\Zit \Phi_{\mu, \epsilon}\subset  \cD_\mu^\epsilon 
\end{align}
The center $C_{\Mc}$ acts  on $ \cD_\mu^\epsilon(\mu+2\rho,\ul{\epsilon}_0 )$ by the character 
\begin{align}\label{cent1}
-\zeta(\mu)= (2\gamma_2+2\gamma_4+\dots+2\gamma_{m-1}  +\frac{3}{2}  \gamma_{m+1})
- d\delta
\end{align}

Finally  investigate the structure of  $ \Hom(  \Lambda^\bullet(\fu_{\Zit}) ,\M_{\lambda, \Zit})).$   The conjugation by the 
matrices $c_{2,i}$ in (\ref{conj}) or better conjugation by  the product $\tilde{c}=\prod_{i:i\text{odd}} c_{2,i}$ provides an identification 
\begin{align}\label{ct}
\tilde{c}:X^*(\KK^{\Mc}\otimes \Zit) \to X^*(T).
\end{align}
Note that $\tilde {c}( e_i) =\gamma_{i}^{^\circ M^{(1)}} $ and for even indices $i$ we have $\tilde c( \gamma_i)=\gamma_i.$

\bigskip
This suggests that we consider  $ \Hom(  \Lambda^\bullet(\fu_{\Z } ,\M_{\lambda, \Z }))$ as a module for $\Mc$
and we even restrict our attention to the action of the center $C_{\Mc}.$ We have the following proposition
which must be already in  \cite{Ko}.
\begin{prop}
The character $ \zeta(\mu)$ occurs only in degree $l(\wu)$ and  the eigenspace
 $$
\BH^{l(\wu)} (\fu_\Z,\MlZ)=\Hom (\Lambda^{l(\wu)}(\fu_\Z ), \M_{\lambda,\Z})(\zeta(\mu))
 $$
 is irreducible with highest weight $w_{\rm un}(\lambda+\rho)-\rho.$ The homomorphism
 $$
 \Hom (\Lambda^{l(\wu)}(\fu_\Z) , \M_{\lambda,\Z})(\zeta(\mu))  \to H^{l(\wu)}(\fu_\Z,\M_{\lambda,\Z})
 $$
 is an isomorphism. 
\end{prop}
\begin{proof}
The Lie algebra $\fu_\Z $ has the basis $e_\beta,$ where $\beta\in  \Delta^+\setminus\{\alpha_1,\alpha_3,\dots,\alpha_m\}. $ Let us denote by $e_\beta^\vee$ the dual basis.  For any subset $J=\{\beta_1,\beta_2,\dots, \beta_s\}\subset   \Delta_{^\circ P}^+$  we define
$e_J^\vee = e^\vee_{\beta_1}\wedge e^\vee_{\beta_2}\wedge\dots e^\vee_{\beta_s}.$   The element 
$e_J^\vee$ is an eigenvector for the standard maximal torus $T\subset{ \Mc},$ the eigenvalue is the character 
$\chi_J=-\sum \beta_i.$
For any Kostant representative $w\in W^{^\circ P} $ we define the set   $\Delta^+(w)=\{\alpha \vert w^{-1} \ag<0 \}.$ 
Then we know that the restriction of $\wu(\lambda+\rho)-\rho =J_{\Delta^+(w)} $ to $C_{^\circ M}$ is 
$\zeta(\mu).$   The weight $ J_{\Delta^+(w)} $ is the highest weight of an irreducible $\Mc$ submodule  $\cN$
in $\Hom (\Lambda^{l(\wu)}(\fu_\Z) , \M_{\lambda,\Z})(\zeta(\mu))$  and the weight subspaces
in $\cN$ are of multiplicity one and of the form $\chi_{J^\prime}.$ Now a simple computation shows
that a subset $J_1\subset  \Delta_{^\circ P}^+ $ for which the restriction of $\chi_{J_1} $ to $C_{^\circ M}$ is equal
to $\zeta(\mu)$ must be on of the $\chi_{J^\prime}$ occurring in  $\cN$ and hence it follows that
$\Hom (\Lambda^{l(\wu)}(\fu_\Z) , \M_{\lambda,\Z})(\zeta(\mu))$ is irreducible.
\end{proof}

This implies that 
 \begin{align}\label{Hodge}
\Hom_{\KK^{\Mc}} (\Lambda^\bullet( \cme_{\Zit}/ \ck_{\Zit}) ,  \cD_\mu^\epsilon  \otimes
   \BH^{l(\wu)}(\fu_\Z, \M_{\lambda,\Z})) =\cr \Hom_{\KK^{\Mc}} (\Lambda^\bullet( \cme_{\Zit}/ \ck_{\Zit}) ,  \cD_\mu^\epsilon  \otimes
  \Hom(\Lambda^\pkt(\fu_\Z/\fk_Z), \M_{\lambda,\Z}))
\end{align}
 If we choose a generator $x_\lambda$ of the highest weight module $\MlZ (\lambda)$ then 
  \begin{align}
 \xi(\wu)\cdot\lambda)=e^\vee_{\Delta^+(\wu)}\otimes \wu x_\lambda \in  \BH^{l(\wu)}(\fu_\Z, \M_{\lambda,\Z}))
\end{align}
is a generator of the highest weight module $  \BH^{l(\wu)}(\fu_\Z, \M_{\lambda,\Z}))(\wu\cdot \lambda)$ it is actually unique up to a sign. We can modify our Borel subgroup $\bar{B}\subset{ \Mc}$ by flipping
into the opposite in some  of the factors. Then the highest weight with respect to such a Borel subgroup
will be 
\begin{align}
\lambda (\wu,\ul\epsilon)= \sum_{i: i \hbox{ odd}}\epsilon_i b_i\gamma_{i}^{^\circ M^{(1)}}   -\zeta(\mu)
\end{align}
where of course again $\epsilon_i=\pm 1$ and the indices $i$ with $\epsilon_i=-1$ tell us where we flipped the Borel subgroup.
To such a weight we have a generating weight vector  $ \xi(\wu\cdot\lambda,\ul \epsilon).$ 
Let us call these weight vectors $ \xi(\wu\cdot\lambda, \epsilon)$ the extremal weight vectors.

We replace the split torus  $T$ by $\tKK^{\Mc}$, these two tori have $ T_{\rm split}=C_{\Mc}$ in common.   Then  we see that that we have   extremal weight 
   spaces 
\begin{align}
 \BH^{l(\wu)}
   (\fu_{\Zit}, \M_{\lambda,\Zit}))(\wu\cdot \lambda)(\tilde{ c}^{-1}( \lambda (\wu,\ul\epsilon))) =\Zit \tilde {c}^{-1}( \xi(\wu\cdot\lambda, \ul\epsilon))
\end{align}
 and on this weight space     the torus $\tKK^{\Mc}$ acts by the 
character 
$$
\tilde{ c}^{-1}( \lambda (\wu,-\ul\epsilon_0))= -b_1 e_1-b_3 e_3-\dots \epsilon_m b_me_m +\zeta(\mu).
$$
 
 \bigskip
$\cD_\mu^\epsilon$    contains the rank one submodule  $  \cD_\mu^\epsilon(\mu_c(\lambda,\epsilon))=\Zit \Phi_{\mu,\ul{\epsilon_0}} $
 and $\tKK^{\Mc}$  acts by the character 
  $$ 
  \mu_c(\epsilon,\lambda)= \sum_{i: i\text{odd},i<m}   (b_i+2)e_i + \epsilon(b_m+2)e_m -\zeta(\mu)
  $$
 and this is a minimal $\KMc$ type (see (\ref{decoD})). The sum of these two characters is 
  $
\nu(\ul\epsilon_0)= 2e_1+2e_2\dots+ \epsilon 2 e_m
 $
 and hence we see 
  \begin{align}\label{Hu1}\begin{matrix}
 \Hom_{\KK^{\Mc}} (\Lambda^\bullet( \cme_{\Zit}/ \ck_{\Zit}) ,  \cD_\mu^\epsilon  \otimes
   \Hom(  \Lambda^\bullet(\fu_{\Zit}) , \Ml))= \cr
   \Hom_{\KK^{\Mc}} (\Lambda^\bullet( \cme_{\Zit}/ \ck_{\Zit})(\nu(\ul\epsilon_0)) ,  \cD_\mu^\epsilon(  \mu_c(\epsilon,\lambda))  \otimes 
   \Hom(  \Lambda^\bullet(\fu_{\Zit}) , \Ml))(\tilde{ c}^{-1}( \lambda (\wu,-\ul\epsilon_0))
   \end{matrix}
\end{align}
Each of the modules in the argument on the right hand side is of rank one and we have chosen a generator for each of them.
Hence we see 
\begin{align}
 \Hom_{\KK^{\Mc}} (\Lambda^\bullet( \cme_{\Zit}/ \ck_{\Zit}) ,  \cD_\mu^\epsilon  \otimes
   \Hom(  \Lambda^\bullet(\fu_{\Zit}) , \Ml)) =\Zit \Omega(\lambda,\epsilon)
\end{align}
where $\Omega(\lambda,\epsilon) $ is the tensor product of the generators and $\epsilon=\pm 1.$
The generator sits in degree $b_n= \rn +l(\wu).$

If $n$ is odd then the choice of $\ul \epsilon$ is irrelevant, if $n$ is even  we get two irreducible $(\fg_\Z, \Ke)$ modules.
As we did in the case $G=\Gl_2$ we can enlarge the connected group scheme $\Ke$ to the larger
group scheme $\KK =\Ke\ltimes \{\eta\}$ where $\eta $ is  the diagonal matrix which has $1$ on the diagonal up the $(n-1)$-th entry and  $-1$ the $n$-th entry. Then $\Dl = \Dl^{(+1)}\oplus \Dl^{(-1)} $ is an irreducible  $(\fg_\Z, \KK)$ module over 
$\Zit,$ the element $\eta$ yields an isomorphism between the two summands. 

Since our weight $\lambda$ is essentially self dual, i.e. $a_i=a_{n-i}$ we have the constraint
$a_{\frac{n}{2}}\equiv 2d\mod 2.$ Then it is clear that 
\begin{align}
\eta \Omega(\lambda,\epsilon) =(-1)^{\frac{a_{\frac{n}{2}}-2d}{2}} \Omega(\lambda,-\epsilon)  
\end{align}
We form again the elements 
\begin{align}\label{genpm}
\omega^{(1)}_\lambda=  \Omega(\lambda,+1)+\Omega(\lambda,-1), \omega^{(2)}_\lambda=  \Omega(\lambda,+1)-\Omega(\lambda,-1)
\end{align}
and then these two elements are the generators for the $\pm $ eigenspaces under the action of $\eta.$
 We have the decomposition 
 \begin{align*}
 \begin{matrix}
H^\pkt(\fg_\Z, \Ke, \D_\lambda\otimes \M_{\lambda,\Zz}) =\cr H_+^\pkt(\fg_\Z, \Ke, \D_\lambda\otimes \M_{\lambda,\Zz} ) \oplus H_-^\pkt(\fg_\Z, \Ke, \D_\lambda\otimes \M_{\lambda,\Zz} )
\end{matrix}
\end{align*}
 and 
 \begin{align*}
H^\pkt(\fg_\Z, \KK, \D_\lambda\otimes \M_{\lambda,\Zz})\iso H_+^\pkt(\fg_\Z, \Ke, \D_\lambda\otimes \M_{\lambda,\Zz} ) .
\end{align*}
 \bigskip
In degree $\pkt=b_n$ the cohomology is the free $\Zz$-module  generated by a class $\omega^{(e)}_\lambda$ where $ e=1, 2$.

  \subsection{The arithmetic of the intertwining operator}
  
  In this section we apply the above considerations to study the arithmetic properties 
  of  an  intertwining operator between two induced modules. 
  
   We   start from the group scheme $\Gl_N/\Z,$ let $B$ be the standard Borel subgroup  and consider the standard parabolic subgroups $P\supset B (\text{ resp.}P^\prime\supset B$) with
  reductive quotient $M=\Gl_n\times \Gl_{n^\prime} =M_1\times M_2 $( \text{ resp. } $ M^\prime=\Gl_{n^\prime}\times \Gl_n.$)  
  Let $U_P \text{ resp.} U_{P^\prime}$ be the unipotent radicals. 
  Let $\pi=\{ \ag_1,\dots ,\ag_{N-1}\} \subset X^*(T)$ be the set of positive (with respect to $B$)   simple roots. We identify the set of simple roots with the set of indices $\{1,2,\dots,N-1\}.$ Let us denote  by $w_N^-$ the permutation which reverses the order by $i\mapsto i^\prime$, i.e. $i^\prime= N-i.$
   Then the positive simple roots for $M$ are $\pi_M=\{\ag_1,\dots, \ag_{n-1}\}\cup \{\ag_{n+1}, \dots, \ag_{N-1}\},$
   and accordingly we denote the system of simple roots of $M^\prime$ by  $\pi_{M^\prime}=w_N^-(\pi_M)=\{\ag_1,\dots, \ag_{n^\prime-1}\}\cup \{\ag_{n^\prime+1}, \dots, \ag_{N-1}\}.$ Let $\g_n \text{ resp. } \g_{n^\prime}$
   be the fundamental weight attached to the missing root $\ag_n\text{ resp. }\ag_{n^\prime}.$ If $\Delta^+_{U_P}
   \text{ resp. } \Delta^+_{U_{P^\prime}}$ are the positive roots occurring in these radicals, let
   $\rho_{U_P} , \rho_{U_{P^\prime}}$ be the half sums over these roots. Then
 \begin{align}
\rho_{U_P} =\frac{N}{2} \gamma_n, \; \rho_{U_{P^\prime}}=\frac{N}{2} \gamma_{n^\prime}
\end{align} 
   
  We choose  a highest weight $\lambda $ for $\Gl_N,$ let 
 $\M_\la$ be the resulting  $\Gl_N-$ module. 
 
  We pick a Kostant representative $w\in W^P,$ and we write
  \begin{align}
\begin{matrix}
  \tilde\mu=w(\lambda +\rho) -\rho = \sum_{i=1}^{n-1} a^\prime_i\gamma_i^{M } + \sum_{i=n+1}^{N-1} a^\prime_i\gamma_i^{M }  +a(w,\lambda) \gamma_n + d \det_N \cr
\cr
=  \sum_{i=1}^{n-1} a^\prime_i\gamma_i^{M }+ d(w,\lambda) \det_n+\sum_{i=n+1}^{N-1} a^\prime_i\gamma_i^{M }+ d^\prime(w,\lambda) \det_{n^\prime} =\mu_1+\mu_2
\end{matrix} 
   \end{align}
Here $\mu_1, (\text{ resp.} \mu_2)$ are highest weights on $M_1  \rs  M_2).$  In some situations it
is more convenient to look at
\begin{align}\label{wlplus}
\tilde\mu+\rho=w(\lambda +\rho)   = \sum_{i=1}^{n-1} b_i\gamma_i^{M } + \sum_{i=n+1}^{N-1} b_i\gamma_i^{M }  +b(w,\lambda) \gamma_n +  d  \det_N
\end{align}
Then we have the relations  $b_i=a^\prime_i+1$ and $b(w,\lambda)=a(w,\lambda)+\frac{N}{2}.$ We define 
the weight of $\tm$:
\begin{align}\label{weight}
\w(\tm) =  \w(\mu_1)+\w(\mu_2)= \sum_{i=1}^{n-1} b_i+ \sum_{i=n+1}^{N-1} b_i  
\end{align}

We make   assumptions on $w:$

\medskip
a)The  length $l(w)=1/2 \dim U_P$- this means that $w$ is balanced.

\medskip
b) 
  Both weights $\mu_1,\mu_2$ are essentially self dual.
  
  \medskip
  c) The weight $\tilde\mu$ is in the negative chamber, this means that $a(w,\lambda)\leq - \frac{N}{2}$
  or $b(w,\lambda)\leq 0.$

\medskip
We have the two longest Kostant representatives $w_P\in W^P \rs w_Q\in W^Q$) which send all the 
roots in $\Delta^+_{U_P}
  ( \text{ resp. } \Delta^+_{U_{P^\prime}})$  into negative roots. If $s_i\in W$ is the reflection attached
    to the simple root $\ag_i$  then  we can write any element $w\in W^P$ as a product of reflections
    $w=s_n s_i s_j\dots s_k.$ We can always complete this product of reflections to get the longest element
    (it always starts with $s_n$ and stops with $s_{n^\prime}$)
    \begin{align}\label{string}
 s_n s_i s_j\dots s_k s_{k^\prime} \dots s_{i^\prime} s_{n^\prime} =w_P
\end{align}
Then $w^\prime=s_{n^\prime} \dots s_{k^\prime}  \in W^Q$  and we
   get  a  one to one
   correspondence between $W^P$ and $W^Q$ which is defined by 
   \begin{align}
w=w_P w^\prime    \text{ or }  w^\prime= w_Q w
\end{align}
  (See 5.3.7) 
   We have $l(w)+l(w^\prime)=\dim U_P,$  since $w$ is balanced we see that $w^\prime$ is also balanced. For $w=e$ the identity element we get $w^\prime=w_Q.$
   
      A presentation of $w_P$ as in (\ref{string}) yields a sequence of roots in $\Delta_{U_P}^+$: The 
   first element in this sequence is $\beta_1=\alpha_n.$ Then we find a root $\beta_2\in \Delta_{U_P}^+$
   such that $s_n\beta_2=\alpha_i $ is a simple root. Then $s_i s_n$ sends the roots $\beta_1 \text { and }
   \beta_2$ into the set of negative roots and we find a root $\beta_3$ such that $s_n s_i \beta_3 $
   is a simple root $\alpha_\nu$ and $s_\nu$ is the next factor in $w_P=s_ns_i s_\nu \dots.$   To say this in different 
   words: We get an ordering $\{\beta_1,\beta_2, \dots, \beta_{d_U}\}=\Delta_{U_P}^+$ such that $x_k^{-1}$
   conjugates exactly   the   first  $k$ roots 
   $\{\beta_1,\dots, \beta_k\}$    into negative roots where
   \begin{align}
x_k=  s_ns_i \dots s_\mu  \text{ is the product of the first }  k \text{ factors in (\ref{string})}.
\end{align}
    
 We can write 
 \begin{align}
w_P=x_1 y_{d_U-1}= s_n(s_i\cdots s_{n^\prime})  =x_2y_2=(s_ns_i)(s_j\cdots s_{n^\prime} )= x_ky_k
\end{align}
  and the $x_k ,y_k$
  are corresponding elements.  We also define a function $r_w:  \{1,2,\dots, ,d_U\} \to \{1,2,\dots, N-1\}= \pi=\{\alpha_1,\dots,
  \alpha_{N-1}\}$ by the rule
  \begin{align}\label{rel}
x_{k }=x_{k-1} s_{r_w(k)}  \text{ or }  x_{k-1}   \alpha_{r_w(k)} =\beta_k
\end{align}

   For our element $w$ above the corresponding element
  $w^\prime $ has also length $\frac{1}{2} \dim_{U_P}$  and we have the corresponding formula 
  \begin{align}
  \begin{matrix}
\tilde{\mu}^\prime=w^\prime
(\lambda +\rho)   = \sum_{i=1}^{n^\prime-1} b_{i^\prime }\gamma_{i ^\prime}^{M^\prime} + \sum_{i= n^\prime+1}^{N-1} b_{i ^\prime}\gamma_{i ^\prime}^{M^\prime}  +b^\prime(w^\prime,\lambda) \gamma_{n^\prime} + d (w^\prime,\la) \det_N\cr
\cr
=  \sum_{i=n+1}^{N-1} b_{i ^\prime}\gamma_{i^\prime }^{M^\prime}+ d (w^\prime,\lambda) \det_{n^\prime}+ \sum_{i=n^\prime+1}^{N-1} b_{i^\prime }\gamma_{i^\prime }^{M^\prime} + d (w^\prime,\lambda) \det_{n } =\mu^\prime_1+\mu^\prime_2.
\end{matrix} 
   \end{align}
The formula tells us that the semi simple components $\mu^{\prime(1)}_1=\mu^{(1)}_{2},\mu^{\prime(1)}_2=\mu^{(1)}_{1}.$ Here we use our assumption that $\mu_i$ are essentially self dual.
   The coefficients of $\gamma_n,\gamma_{n^\prime}$ are related by 
   \begin{align}
 d (w,\lambda) + d(w^\prime,\lambda) =-N
\end{align}
                   
The weights $\mu_1,\mu_2 $ yield Harish-Chandra modules $\D_{\mu_1}( \text{ for } \Gl_{n})$ 
and $\D_{\mu_2}$ for $\Gl_{n^\prime}.$  (See section \ref{Dle} ) and hence  ${\tm}=\mu_1+\mu_2$ provides 
a Harish-Chandra module $\D_{\tm} $ for $M.$ By the same argument we
get a Harish-Chandra module $\D_{{\tmp} }$ for $M^\prime.$  These two $(\fm, \KK^M)$ (resp. $(\fm^\prime,\KK^{M^\prime})$ ) modules both have a minimal $\KK^M$ type $\mu_c(\epsilon,\tilde\mu)=
\mu_c(\epsilon_1,\mu_1)+\mu_c((\epsilon_2,\mu_2)$ (resp. $\KK^{M^\prime}$ type $\mu_c(\epsilon^\prime,\tilde\mu^\prime)).$
    
We consider  the $(\fg_\Z, \Ke)$  modules $   \fI_{ P}^G   \D_\tm \text{ and }     \fI_{  Q}^G   \D_{\tmp}. $   
Both these modules contain the irreducible module $\vartheta_{\mu_c(\epsilon,\tilde\mu)}$ as minimal $\Ke$
type and this $\Ke$-type has multiplicity one.

We take the base extensions to  $\C$ and twist them by a holomorphic variable.  We introduce the abbreviating
notation
$$ \fI_{ P}^G   \D_\tm\otimes \C (\V{\gamma_n}^z) =\fI_{ P}^G   \D_\tm\otimes (z)$$
Then we can write down
the usual intertwining operator
\begin{align}\label{Intz}
T^{w_P,\rm st}(z):    \fI_{ P}^G   \D_\tm\otimes \C (\V{\gamma_n}^z) \to   \fI_{  Q}^G   \D_{\tmp}\otimes \C
(\V{\gamma_{n^\prime}}^{-z} )
\end{align}

which is given by  the integral

 $$ \{ g\mapsto  f(g)\}\mapsto  \{g \mapsto    \int_{U_P(\R)}  f(w_Pug) d_\infty u\} $$ 
 where  $d_\infty u$ 
is the Haar measure obtained from the epinglage. This integral converges if $ \Re(z)>>0$. The action
of $\Ke(\R)$ on these modules is independent of $z$  especially it it clear that the above lowest $\Ke(\R)$-type
occurs in the deformed modules with multiplicity one.

\begin{satz}\label{maintheorem} 
i) This operator extends to a meromorphic operator in the entire $s$-plane and it is holomorphic 
at $z=0. $ 

ii)   At $z=0$  it is an isomorphism and modified by the factor $\frac{1}{\pi^{d_U/2}}$ it  is defined over $\Q $, i.e.
we get an isomorphism
\begin{align}
\frac{1}{\pi^{d_U/2}}T^{w_P,\rm st}(0):    \fI_{ P}^G   \D_{\tm}\otimes \Q  \iso    \fI_{  Q}^G   D_{\tmp}\otimes \Q 
\end{align}
 \end{satz}
 \begin{proof} 
  The first assertion is  over $\C$ and follows from results of Speh (See \cite{Speh}). 
  
   For the second we use the standard strategy and write the operator as a product of operators induced from
 intertwining operators   on $\Gl_2.$ Here we have to deal with the problem that some of the operators
 will not be defined because we encounter poles in the backwards operators.
 
 We apply the consideration from section (\ref{Dl}) to $M$ and $M^\prime$, i.e. to the two factors $\Gl_n$ and 
 $\Gl_{n^\prime}$. Especially we introduce the subgroup scheme $\Mc\subset M$ as the product 
 of the two corresponding groups in the two factors.  This also yields the element $w^M_{\rm un}$ in the Weyl group of $M.$  The module $\D_{\tm}$ is  induced from a module
 $\cD_{w^M_{\rm un}\cdot \tm+ 2\rho_{\Mc}}$ and this module is defined as a submodule from $\fI_{\bar B}^{\Mc}w^M_{\rm un}\cdot (\tm+ 2\rho_{\Mc}).$ Hence we get  $\D_{\tm} \subset \fI_{B_M}^M w^M_{\rm un} \cdot \tm+ 2\rho_{\Mc}$ and finally 
 \begin{align}
\fI_P^G \D_{\tm} \into  \fI_B^G  w^M_{\rm un}\cdot( \tm+ 2\rho_{\Mc}).
\end{align}
 Now we can try to extend the intertwining operator $T^{w_P,\rm st}(0) $ to this larger module. This may not be possible. Therefore we apply the usual technique and deform the induced module by a character $\V{\ul \gamma}^{\ul z}=
  \prod_{i=1}^{N-1} \V{\gamma_i}^{ z_i } $ where $\ul z\in \C^{N-1}.$  We consider the extension
 \begin{align}\label{Intext}
  T^{w_P,\rm st}(\ul z) :   \fI_B^G  w^M_{\rm un}\cdot( \tm+ 2\rho_{\Mc})_\R \otimes\gzz\to
   \fI_B^G  w^M_{\rm un}\cdot( \tmp+ 2\rho_{\Mc^\prime})_\R \otimes \gzp
\end{align}
which is given by the following integral:  We write $U_P$ as product of one parameter subgroups
(note that $\alpha_{n^\prime}=\beta_{d_U}, \alpha_n=\beta_1)$
\begin{align}
U_P  = U_{\alpha_{n^\prime}} \times  U_{\beta_{d_U-1}}\times \cdots \times  U_{\alpha_n}
\end{align}
  \begin{align}
  T^{w_P,\rm st}(\ul z) (f)(g) =  \int_{U(\R)} f(w_Pu g) du  =\int \dots \int f(w_P u_{\alpha_{n^\prime}} \dots u_{\alpha_n} g)du_{\alpha_{n^\prime} }\cdots d u_{\alpha_n}
\end{align}
where the measure is the Lebesgue measure which is normalized by the epinglage.
This integral converges for $\Re(z_i)>> 0$  and has a meromorphic extension into the entire complex plane.

 We have by definition
  \begin{align}
 w_P u_{\alpha_{n^\prime}}u_1= s_n u_{\alpha_n} y_{d_U-1}  u_1,
\end{align}

  where $ u_1\in 
 U_1(\R)= U_{\beta_{d_U-1}}\times \cdots U_{\alpha_{i^\prime}}.$ Hence our integral becomes 
 \begin{align}\label{Intstep}
  T^{w_P,\rm st}(\ul z) (f)(g)= \int_{U_1(\R)}    \int_{U_{\alpha_{n }}(\R)} f(s_n u_{\alpha_n} y_{d_U-1}  u_1g ) d u_{\alpha_n} d u_1
\end{align}
The inner integral is an intertwining operator.   We write our induced modules now as induced from characters $\chi$
on the  $\R$ valued points of the torus $T(\R),$ let
 $\chi =w_{\rm un}\cdot( \tm + 2\rho_{\Mc })_\R)$
then
\begin{align}\label{Ia}
\{ g   \mapsto f(g)\} \mapsto  \{g \mapsto \int_{U_{\alpha_n}(\R)}f(s_n u_{\alpha_n}g) du_{\alpha}  \} 
\end{align}
is an intertwining operator 
\begin{align}
T^{\rm st}(s_n, \chi, \ul z) : I_B^G \chi\otimes \gzz
\to I_B^G( s_n \cdot( \chi\otimes \gzz)).
\end{align}
 where of course $s_n\cdot( \chi\otimes \gzz) = s_n(\chi\otimes \gzz) + s_n(\V{\rho})- \V{\rho}.$
 
 This intertwining operator  is now   induced  from an intertwining operator between two $\Sl_2$
modules. Let $\tilde {H}_{\alpha_n}$
 be the reductive subgroup $H_{\alpha_n}\cdot T,$ the group scheme $\tilde {H}_{\alpha_n}$ is then the 
 Levi quotient of a parabolic subgroup $P_{\alpha_n}.$ Let $B_{\alpha_n} $be the Borel subgroup in $\tilde{H}_{\alpha_n}.$
 Then the integral in (\ref{Ia}) also defines an intertwining operator 
 \begin{align}
 T^{\rm st}_{\alpha_n}(s_n, \chi,\ul z):I_{B_{\alpha_n}}^{{\tilde {H}_{\alpha_n}}}\chi \otimes \gzz\to I_{B_{\alpha_n}}^{{\tilde {H}_{\alpha_n}}}(s_n\cdot( \chi\otimes \gzz))\}.
\end{align}
 
  Our two induced modules can by written as two step induction 
 \begin{align}
I_B^G  \chi \otimes \gzz=  I_{P_{\alpha_n}}^G \;I_{B_{\alpha_n}}^{\tilde {H}_{\alpha_n}}(\chi\otimes \gzz) ,
I_B^G s_n\cdot(  \chi\otimes \gzz) =  I_{P_{\alpha_n}}^G\;  I_{B_{\alpha_n}}^{\tilde {H}_{\alpha_n}}(s_n\cdot(\chi\otimes \gzz))  
\end{align}
and   then our intertwining operator is induced
\begin{align}
T^{\rm st}(s_n,\chi, \ul z) =  I_{P_{\alpha_n}}^G  T^{\rm st}_{\alpha_n}(s_n, \chi,  {\ul z}).
\end{align}
Hence we can write equation (\ref{Intstep})
\begin{align}
  T^{w_P,\rm st}(\ul z) (f)(g) = \int_{U_1(\R)}      I_{P_{\alpha_n}}^G ( T^{\rm st}_{\alpha_n}(s_n, \chi  , \ul z) (f)(y_{d_U-1}  u_1g )d u_1
\end{align}
We iterate this process. We have $y_{d_U-1}=s_i s_j\cdots s_{n^\prime} ,$  we apply the above
process again and eventually we  get 
\begin{align}\label{product}
\begin{matrix} T^{w_P,\rm st}(\ul z) (f)  =   I_{P_{\alpha_{n^\prime}}}^G  T^{\rm st}_{\beta_{d_U-1} }(s_{n^\prime},x^{-1}_{d_U-1}\cdot \chi,x^{-1}_{d_U-1} \ul z  ) \circ\dots\cr
\circ  I_{P_{r_w(k)}}^G T^{\rm st}_{r_w(k)} (s_{r_w(k)},x^{-1}_{k-1}\cdot\chi,x^{-1}_{k-1}(\ul z))\circ \cdots\circ
 I_{P_{\alpha_n}}^G T^{\rm st}_{\alpha_n}(s_n,\chi,\ul z)(f)
 \end{matrix}
\end{align}  
We understand the  $\Gl_2$ intertwining operators 
\begin{align}
T^{\rm st}_{r_w(k)} (s_{r_w(k)},x^{-1}_{k-1}\ul z)): I_{B_{r_w(k)}}^{\tilde{H}_{r_w(k)}}  (x_{k-1}^{-1}\cdot \chi\otimes\gzz) \to
   I_{B_{r_w(k)}}^{\tilde{H}_{r_w(k)}} (x_{k }^{-1}\cdot \chi\otimes\gzz )
\end{align}
In section (\ref{IO}) we defined the algebraic intertwining operator 
$T^{alg}_{\lambda^-}$  
  by fixing their value on he lowest $\KK$ type.   If we put $\chi=\lambda^-_\R$ then we can extend these
  operators to the twisted modules 
  \begin{align*}
 T^{alg}(s_\ag,\chi, z): I_{B_{\ag}}^{\Gl_2} \chi\otimes\V{\gamma_\alpha}^z \to  I_{B_{\ag}}^{\Gl_2} s_\ag\cdot( \chi\otimes\V{\gamma_\alpha}^z)
 \end{align*} 
      
  With respect to the basis  $\Phi_\nu$   the operator $ T^{alg}(s_\ag,\chi, z)$ acts as a diagonal matrix with entries in the rational function field $\Q(\ul z)$ and we know the factor that compares $ T^{alg}(s_\ag,\chi, z)$ to $ T^{\rm st}(s_\ag,\chi, z),$
it is a ratio of $\Gamma$ values. To compute $T_{r_w(k)} (s_{r_w(k)},x^{-1}_{k-1}\cdot\chi)$ we need to know the restriction of 
$x^{-1}_{k-1}\cdot\chi$ to the torus $T_{r_w(k)} ,$ we recall that  the coroot $\alpha^\vee_{r_w(k)}:\Gm \to T_{r_w(k)}$
provides an identification. Hence the restriction of $x^{-1}_{k-1}\cdot\chi$ to $T_{r_w(k)}$ is a character on
$\Gm(\R)=\R^\times$ and an easy computation shows that this character is
\begin{align*}
t  \mapsto  t^{<\ak^\vee, x_{k-1}\chi>} \V{t}^{<\ak^\vee, x_{k-1}^{-1} \rho -\rho>+ <\ak^\vee,\ul z>}
\end{align*}
 We still can manipulate the exponent. We have  $x_{k-1}\ak^\vee= \beta_k^\vee. $  Then the first exponent becomes 
 $<\beta^\vee_k, \chi>$ and for the second one we get $<\beta^\vee -\ak^\vee,\rho>= h(\beta_k),$ where
 for $\beta= \alpha_\nu+\dots+ \alpha_{\nu+h}$ we put $h(\beta)=h.$

Hence  our character is
\begin{align}
t  \mapsto  t^{<\beta_k^\vee, \chi>} \V{t}^{h(\beta_k) + <\ak^\vee,\ul z>}
\end{align}
Then we put $\epsilon_w(k,\chi)=0$ if  $<\beta_k^\vee, \chi >\equiv 0\mod 2$ and $\epsilon_w(k,\chi)=1$ 
else. Then we get from our formulae in  section \ref{THC}
\begin{align}\label{T1}
T^{\rm st}_{r_w(k)} (s_{r_w(k)},x^{-1}_{k-1}\cdot\chi,x^{-1}_{k-1}\ul z)
  =\frac{\Gamma(\frac{<\beta_k^\vee, \chi>+\epsilon_w(k,\chi)+h(\beta_k)-1 +<\ak^\vee,\ul z>}{2})} {\Gamma(\frac{ <\beta_k^\vee, \chi>+\epsilon_w(k,\chi)+h(\beta_k)  +<\ak^\vee,\ul z>}{2})}\Gamma(\frac{1}{2}) M_k(\ul z)
\end{align}
where $M_k(\ul z) $ is a diagonal matrix with entries in the field of rational functions $\Q(\ul z),$  it may have a pole
at the hyperplane   $<\ak^\vee,\ul z>=0 $ but
the ratio
\begin{align*}
M^*_k(\ul z)=\frac{M_k(\ul z)}{\Gamma(\frac{ <\beta_k^\vee, \chi>+\epsilon_w(k,\chi)+h(\beta_k)  +<\ak^\vee,\ul z>}{2})} 
\end{align*}
is holomorphic on this hyperplane   and hence can be evaluated at $\ul z=0.$

\bigskip

By definition the number $<\beta_k^\vee, \chi>+\epsilon_w(k,\chi)$ is even  and hence
this operator is holomorphic at $\ul z=0$ if $h(\beta_k)$ is  even. In this case   the character
restricted to the torus $T_{r_w(k)}$ is cohomological.We find
\begin{align}
T^{\rm st}_{r_w(k)} (s_{r_w(k)},x^{-1}_{k-1}\cdot\chi,x^{-1}_{k-1}\ul z)\vert_{\ul z=0}
  = {\Gamma(\frac{<\beta_k^\vee, \chi >+\epsilon_w(k,\chi)+h(\beta_k)-1  }{2})} \Gamma(\frac{1}{2}) M^*_k(0)  
\end{align}
where   $M^*_k(0)$ is a matrix with rational entries and the factor in front is $\pi\times$ a rational number.
This tells us that 
\begin{align}
T^{\rm st}_{r_w(k)} (s_{r_w(k)},x^{-1}_{k-1}\cdot\chi)\vert_{\ul z=0}=  \pi\times  T^{alg}_{r_w(k)} (s_{r_w(k)},x^{-1}_{k-1}\cdot\chi) q_k, \end{align}
where $q_k\in \bar\Q^\times.$

If $h(\beta_k) $ is odd  then the hyperplane $<\ak^\vee,\ul z>=0$
may be a first order pole, this happens exactly when

\begin{align*}
<\beta_k^\vee, \chi_{alg}>+\epsilon_w(k,\chi)+h(\beta_k)-1=0,-2,-4,\dots 
\end{align*}
 We put $m_k=1 $ if $h(\beta_k)$ is odd and we encounter a pole and $m_k=0$ else. Then we manipulate the right hand side  in equation (\ref{T1}) and change it to
 \begin{align}
<\ak,\ul z>^{m_k}\frac{\Gamma(\frac{<\beta_k^\vee, \chi_{alg}>+\epsilon_w(k,\chi)+h(\beta_k)-1 +<\ak^\vee,\ul z>}{2})} {\Gamma(\frac{ <\beta_k^\vee, \chi_{alg}>+\epsilon_w(k,\chi)+h(\beta_k)  +<\ak^\vee,\ul z>}{2})}\Gamma(\frac{1}{2}) \frac{M_k(\ul z)}{<\ak,\ul z>^{m_k}}
\end{align}
the last factor to the right is still a a diagonal matrix with entries in the field $\Q(\ul z).$ The expression 
in values of the Gamma-function can be evaluated at $\ul z=0$ and the result is a rational number, the
two contributions of $\sqrt{\pi}$ cancel.

We return to our factorization of the intertwining operator $T^{w_P,\rm st}(\ul z).$ It is an intertwining operator 
between two Harish -Chandra modules with a $\Q$ structure.  They have a decomposition into  $ \Ke$
types (which are of course $\Q$ vector spaces $ \otimes \C$.) We consider the restriction to a $ \Ke$
type $\vartheta$ which is of course finite dimensional. Then our product decomposition yields
\begin{align}
\bigl(\prod_{k } \frac{<\ak,\ul z>^{m_k}\Gamma(\frac{<\beta_k^\vee, \chi >+\epsilon_w(k,\chi)+h(\beta_k)-1 +<\ak^\vee,\ul z>}{2})} {\Gamma(\frac{ <\beta_k^\vee, \chi >+\epsilon_w(k,\chi)+h(\beta_k)  +<\ak^\vee,\ul z>}{2})}\Gamma(\frac{1}{2})\bigr)
M(\vartheta,\ul z)
\end{align} 
where
\begin{align*}
 M(\vartheta,\ul z) \in \Hom_{\Ke}(  \fI_B^G  w^M_{\rm un}\cdot( \mu+ 2\rho_{\Mc})(\vartheta) ,
  \fI_B^G  w^{M^\prime}_{\rm un}\cdot( \mu^\prime+ 2\rho_{\Mc^\prime})(\vartheta))\otimes\Q(\ul z)
 \end{align*}
 The factor in front  can be evaluated at $\ul z=0.$   Each factor  contributes by a non zero rational number    
or $\pi$ times a non zero rational number. We get a factor $\pi$ in the cases where  $h(\beta_k)$ is even
and this happens $d_U/2$ number of times. So we would be finished with the proof if we evaluate
$ M(\vartheta,\ul z)$ at $\ul z=0$ and observe that this is a matrix with entries rational numbers. But
we do not know whether $M(\vartheta,\ul z)$ can be evaluated at zero, we have moved the poles
in  the Gamma-factors into  $M(\vartheta,\ul z),$   the extended intertwining operator in equation 
(\ref{Intext}) may not be regular at $\ul z=0.$

We are only interested in the restriction of the operator to  $\fI_P^G \D_\tm\otimes \C[\V{\gamma_n}^z].$
(See  (\ref{Intz}) , i.e. we restrict it to the line $z \V{\gamma_n}\subset  \C^{N-1}.$    We notice 
that this line is not contained in any of the hyperplanes $<\ak^\vee,\ul z> =k\in \Z.$  Hence
we see that our operator $M(\vartheta,z)$ is a meromorphic function in the variable $z.$

 The modules $ \fI_P^G \D_\tm $ and $\fI_Q^G {\tmp}$ contain  the special irreducible 
$\Ke$-module $\cX[\tm]$ with highest weight $ {\mu_c(\epsilon,\tilde\mu)}$ with multiplicity one. This $\Ke$ module occurs with higher 
multiplicity $t$  in $\fI_B^G  w^M_{\rm un}\cdot( \tm+ 2\rho_{\Mc})\otimes\C\V{\gamma_n}^z$. Restriction
to this $\Ke$ type yields  a diagram 
\begin{align}
\begin{matrix}
  \cX[\tm])\otimes (z) & \ppfeil{T^{w_P,\rm st}(z)} &    \cX[\tm])\otimes (-z) \cr
\downarrow && \downarrow\cr
( \cX[\tm])^t\otimes (z) &\ppfeil{T^{w_P,\rm st}(z)}& ( \cX[\tm]) ^t \otimes(-z)  
\end{matrix}
\end{align}
where the  downarrows are   the inclusion by the first coordinate.  
Then our  matrix $M(\vartheta_{\mu_c(\epsilon,\tilde\mu)},  z)$ will be an $t\times t$ matrix with entries $C_{l,m}( \vartheta_{\mu_c(\epsilon,\tilde\mu)}, z)$
where  $C_{l,m}(\vartheta_{\mu_c(\epsilon,\tilde\mu)}, z)\in \Q( z).$ We look at this first row, which tells us what happens 
to the first coordinate under $T^{w_P,\rm st}( z). $ This first row is $(C_{1,1}(\vartheta_{\mu_c(\epsilon,\tilde\mu)},  z), 0 \dots, 0).$  
The  rational function  $(C_{1,1}(\vartheta_{\mu_c(\epsilon,\tilde\mu)},  z)\in \Q(z)$  is regular at $ z=0.$ 
(See \cite{ha-ra}, Prop. 7.44).
Therefore we can evaluate the first row at $  z=0.$ The result will be
$((C_{1,1}(\vartheta_{\mu_c(\epsilon,\tilde\mu)},0),  0, \dots, 0 ))$ where $C_{1,1}(\vartheta_{\mu_c(\epsilon,\tilde\mu)},0)\in \Q^\times.$

 \end{proof}
It is clear that the operator $\frac{1}{\pi^{\frac{d_U}{2}}}T^{w_P,\rm st}( 0) $ induces an isomorphism in cohomology

\begin{align}\label{IsoQ}
T^{w_P,\bullet}  : H^\pkt (\fg_\Z  ,\KK,  \fI_{ P}^G   \D_{\tm}\otimes \M_\la )\otimes \Q   \to  H^\pkt (\fg_\Z  ,\KK,  \fI_{ Q}^G   \D_{\tmp} \otimes \M_\la )  \otimes \Q
\end{align}  

We can compute these cohomology groups using Delorme. Let $\KK^{(1),M}  \subset M $ be the  connected maximal definite  group scheme then

$$ \begin{matrix} H^\pkt (\fg_\Z  ,\KK,  \fI_{ P}^G   \D_{\tm}\otimes \M_\la )   \into  \cr H_+^\pkt ( \fm, \KK^{(1),M} , \D_{\tm}\otimes \cM_{\wu\cdot\lambda}) 
\oplus H_-^\pkt ( \fm, \KK^{(1),M} , \D_{\tm}\otimes \cM_{\wu\cdot\lambda}) 
\end{matrix}
$$
and the same holds for the $Q$ part
$$
 \begin{matrix}H^\pkt (\fg_\Z  ,\KK,  \fI_{ Q}^G   \D_{\tmp}\otimes \M^\prime_\la )   \into \cr  H_+^\pkt ( \fm, \KK^{(1),M^\prime} , \D_{\tmp}\otimes \cM_{\wu\cdot\lambda}) 
\oplus H_-^\pkt ( \fm, \KK^{(1),M^\prime} , \D_{\tmp}\otimes \cM_{\wu\cdot\lambda}) 
\end{matrix}
$$
and the inclusion is always an isomorphism to the $+$ component. Now we remember that $M=M_1\times M_2$
one of the factors is a $\Gl_n$ with $n$ even the other is $\Gl_{n^\prime} $ with $n^\prime $ odd.  Then (in the lowest degree)  the 
cohomology $H_+^\pkt ( \fm, \KK^{(1),M } , \D_{\mu }\otimes \cM_{\wu\cdot\lambda}) $ is generated 
by an element $\omega_{\mu_1}^{(e)}\otimes \omega_{\mu_2}$ where $e= 1\text{ or } 2.$ The cohomology 
$H_+^\pkt ( \fm, \KK^{(1),M^\prime} , \D_{\tmp}\otimes \cM_{\wu^\prime\cdot\lambda})  $  is 
generated by $\omega_{\mu^\prime_2} \otimes \omega_{\mu^\prime_1}^{(e^\prime)}$ where $1^\prime=2, 2^\prime=1.$
(This introduces the relative period). Then we get 
 
\begin{align}
 T^{w_P,b_n+b_{n^\prime}} ( \omega_{\mu_1}^{(e)}\otimes \omega_{\mu_2}) = C_{1,1}(\vartheta_{\mu_c(\epsilon,\tilde\mu)},0)\omega_{\mu^\prime_2} \otimes \omega_{\mu^\prime_1}^{(e^\prime)}
 \end{align}
 
 This rationality result is applied in \cite{ha-ra} to prove a rationality result for ratios of critical values
 of Rankin-Selberg $L$ -functions at consecutive critical arguments. We recall from the work of Shahidi \cite{Shahidi}
 that we can attach a local $L$ function $L^{\rm coh}_\infty(\D_{\tm},s)$ to our (representation) Harish-Chandra module 
 $\D_{\tm}.$  (This local $L$-function differs by the usual shift from Shahidis $L$ function.)  Then we can rewrite 
 the formula above into
 \begin{align}
T^{w_P,\rm st}(0)( \omega_{\mu_1}^{(e)}\otimes \omega_{\mu_2}) =c_\infty(\tm) \frac{L^{\rm coh}_\infty(\D_{\tm},\w(\tm)+b(w,\lambda))}{L^{\rm coh}_\infty(\D_{\tm},\w(\tm)+b(w,\lambda)+1)} ( \omega_{\mu^\prime_2} \otimes  \omega_{\mu^\prime_1}^{(e^\prime)})
\end{align}
The local $L-$ function can be expressed in terms of products of functions
  $\Gamma_\C(z)=2(2\pi)^{-s} \Gamma(s) $ (See \cite{ha-ra} ,7.2.1 ) and using this expression we find
  \begin{align}
\frac{L^{\rm coh}_\infty(\D_{\tm},\w(\tm)+b(w,\lambda) )}{L^{\rm coh}_\infty(\D_{\tm},\w(\tm)+b(w,\lambda)+1 )} =\frac{\pi^{d_U/2}}{\prod N_i(w,\tm)}
\end{align}
where $N_i(w,\tm)$ are certain integers (See Cor. 7.33 in \cite{ha-ra}). The combinatorial lemma in \cite{ha-ra} (Appendix by Weselmann)
implies that under the given conditions the   $N_i(w,\tm)\not= 0.$ 

  Our rationality result is then equivalent to the assertion 
that  $c_\infty(\tm)\in \Q^\times.$ It enters in the proof of the main theorem in \cite{ha-ra}.

\subsection{An intriguing question} 
It is certainly possible to prove the    necessary rationality result at the place infinity with somewhat lesser
effort. In \cite{ha-ra} the  relative period is defined after we make some choices of basis vectors in various 
vector spaces (most of the time one dimensional). Then the computation of the intertwining operator
comes down to see its effect on these basis elements and this computation can be carried out quite directly.

We develop the concept  of Harish-Chandra modules over $\Z$ because this gives us some motivation for the
choice of these basis elements. But we can get more profit out of it. We have seen that  the cohomology
modules 
\begin{align}
H^\pkt (\fg_\Z  ,\KK,  \fI_{ P}^G   \D_{\tm}\otimes \M_\la )\otimes \Z[\frac{1}{2}]   ,\; H^\pkt (\fg_\Z  ,\KK,  \fI_{ Q}^G   \D_{\mu^\prime} \otimes \M_\la )  \otimes \Z[\frac{1}{2}] 
\end{align}
in lowest degree are free of rank one and the  intertwining multiplied by $1/\pi^{d_U/2}$ induces an isomorphism
if we tensorize by $\Q.$  But we may also consider the slightly modified operator
\begin{align}
\tilde{T}^{w_P} (\tm)=\frac{L^{\rm coh}_\infty(\D_{\tm},\w(\tm)+b(w,\lambda)+1)}{L^{\rm coh}_\infty(\D_{\tm},\w(\tm)+b(w,\lambda) )})T^{w_P,\rm st}(0)
\end{align}
which also induces an isomorphism between the two modules after we tensor them by $\Q.$  We ask the question

 {\it Is the modified operator }
 \begin{align}
\tilde{T}^{w_P} (\tm): H^\pkt (\fg_\Z  ,\KK,  \fI_{ P}^G   \D_{\tm}\otimes \M_\la )\otimes \Z[\frac{1}{2}] \to H^\pkt (\fg_\Z  ,\KK,  \fI_{ Q}^G   \D_{\tm^\prime} \otimes \M_\la )  \otimes \Z[\frac{1}{2}] 
\end{align}
{\it an isomorphism ?}

This is of course equivalent with the assertion $c_\infty(\tm)\in \Z[\frac{1}{2}]^\times.$
The only non trivial case where we know 
that this true is 
the case $N=3.$ (See \cite{Za}).

A similar question is discussed in my preprint "Secondary Operations in the Cohomology of
Harish-Chandra Modules" (\cite{ha-ra}, folder "Eisenstein", SecOps.pdf)

   \subsection{Fixing the periods}
  In \cite{ha-ra} the authors prove a rationality result for   ratios of consecutive  special values of Rankin-Selberg
  $L-$ functions. In this rationality result a certain relative period $\Omega(\tilde{\sigma}_f)$ enters, this period  is 
  basically a non zero complex number which is defined modulo $E^\times$ where $E$ is a number  field
  over which $\tilde{\sigma}_f$ is defined. We will show here that we can make this choice of periods more
  precise so that they are essentially defined modulo  the units $\cO_E^\times.$ (For a more precise statement
  see further down.)  This allows us to speak of the prime factorization of the ratios of critical values (divided by the period)
  and this is    of arithmetic interest. 
    These considerations are not included into \cite{ha-ra}  because the authors where concerned that the paper
  may become too long.

   Assume $n$ even and $G=\Gl_n/\Z,$ for simplicity we assume $F=\Q.$ We consider the  inner  cohomology
   $H_{!}^\pkt(\SGK, \M_{\lambda,\Z}).$ This is a finitely generated $\Z$ module   $H_{!,\ganz}^\pkt(\SGK, \M_{\lambda,\Z}) $ 
   is its quotient by torsion. We have an action of the "integral" Hecke algebra on these cohomology groups (See \cite{harder-book},  Chap.3, 2.3. ) If we extend $\Q$ to  finite extension $E/\Q$ and tensor by the ring of integers $\cO_E$ then we get a decomposition 
   up to isogeny (See \cite{harder-book} , 2.3.9.)
   \begin{align}
H^\pkt_{!,\ganz}(\SGK, \M_{\lambda,\cO_E}) \supset \bigoplus_{\pi_f\in \Coh(G,\lambda,K_f)} H^\pkt_{!,\ganz}(\SGK, \M_{\lambda,\cO_E})(\pi_f)
\end{align}
We have the action of $\pi_0(\Gl_n(\R))$ on these cohomology groups and after inverting 2 we get an isomorphism
\begin{align*}
 H^\pkt_{!,\ganz}(\SGK, \M_{\lambda,\cO_E[\frac{1}{2}]})(\pi_f)(+)\oplus  H^\pkt_{!,\ganz}(\SGK, \M_{\lambda,\cO_E[\frac{1}{2}]})(\pi_f)(-)
\end{align*}
If a summand $\pi_f$ is strongly inner (See \cite{ha-ra}, 5.1)  and if we tensor by  $\Q$  the two summands become isomorphic 
$\HH_{K_f}^G$ modules.  If $S$ is a finite set of primes containing the primes where $K_f$ is ramified
then 
   $\HH_{K_f}^{G,S}=\prod_{p\not\in S}\HH_{K_p}^G$
  is a central sub algebra of $\HH_{K_f}^G.$   (See \cite{harder-book}, 2.3.2).   We say that $\pi_f$ is weakly split by $E$
  if the restriction of $\pi_f$ to $\HH_{K_f}^{G,S}$ is a homomorphism $\psi^S(\pi_f) : \HH_{K_f}^{G,S} \to \cO_E.$
  (The eigenvalues of the Hecke operators outside $S$ lie in $\cO_E.$)  We define   
    \begin{align*}
    \begin{matrix}
 H^\pkt_{!,\ganz}(\SGK, \M_{\lambda,\cO_E[\frac{1}{2}]})(\psi(\pi_f),\epsilon)=\cr
 \{\xi\in    H^\pkt_{!,\ganz}(\SGK, \M_{\lambda,\cO_E[\frac{1}{2}]})(\epsilon)\;\vert\;
 h\xi=\psi(\pi_f)(h)\xi \text { for all } h\in \HH_{K_f}^{G,S}\}.
\end{matrix}
\end{align*}
Since our group is $\Gl_n$
  we have strong multiplicity one.  It follows that the isomorphism type $\pi_f$ is uniquely determined 
  by $\psi(\pi_f)$ and it  is absolutely irreducible, more precisely 
  \begin{align}\label{piS}
 H^\pkt_{!,\ganz}(\SGK, \M_{\lambda,\cO_E[\frac{1}{2}]}) (\pi_f ,\epsilon)= H^\pkt_{!,\ganz}(\SGK, \M_{\lambda,\cO_E[\frac{1}{2}]})(\psi(\pi_f),\epsilon)
\end{align} 
If we define $E(\pi_f)\subset  E$ to be the subfield of $E$ which is generated by the values $\psi(\pi_f)(h)$ (it is independent of the choice of $S$) then
\begin{align}
 H^\pkt_{!,\ganz}(\SGK, \M_{\lambda,\cO_E[\frac{1}{2}]})(\pi_f)(\epsilon)= H^\pkt_{!,\ganz}(\SGK, \M_{\lambda,\cO_{E(\pi_f)}[\frac{1}{2}]})(\pi_f)(\epsilon)\otimes_{\cO_{E(\pi_f)}[\frac{1}{2}]}\cO_E.
\end{align}
 
  If $\pi_f$ is strongly inner then the cohomology  modules in lowest degree
  \begin{align*}
H_{!!}^{b_n}(\SGK, \M_{\lambda,\cO_E})(\pi_f,\epsilon)
\end{align*}
  are absolutely irreducible $  \HH_{K_f}^G$ modules (See \cite{ha-ra},3.3.3)  (This also means that  the homomorphism
  $\HH_{K_f}^G\to \End_E(H^{b_n}(\SGK, \M_{\lambda,E})(\pi_f,\epsilon))$ is surjective.) The module of homomorphisms
  \begin{align}
{\cal T}^{\rm alg}(\pi_f,\epsilon)= \Hom_{\HH_{K_f}^G}(H_{!!,\ganz}^{b_n}(\SGK, \M_{\lambda,\cO_E})(\pi_f,\epsilon),  H_{!!,\ganz}^{b_n}(\SGK, \M_{\lambda,\cO_E})(\pi_f,-\epsilon))
\end{align}
  is a finitely generated, torsion free $\cO_E$ module of rank one. We consider it as an invertible sheaf for the
  Zariski topology  
  on $\Spec(\cO_E).$ For any open subset $U\subset  \Spec(\cO_E)$ we use the usual notation ${\cal T}^{\rm alg}(\pi_f,\epsilon)(U)$  for the module of sections over $U,$ this is a module for $\cO(U).$ It is clear 
that we   find a covering of $\Spec(\cO_E)$ by two open sets $U_1, U_2$ such that 
$${\cal T}^{\rm alg}(\pi_f,\epsilon)(U_i)=
  \cO_E(U_i) T_i^{\rm alg}(\pi_f,\epsilon)$$ where 
   $$T_i^{\rm alg}(\pi_f,\epsilon)\in \Hom_{\HH_{K_f}^G}(H_{!!,\ganz}^{b_n}(\SGK, \M_{\lambda,\cO_E})(\pi_f,\epsilon),  H_{!!,\ganz}^{b_n}(\SGK, \M_{\lambda,\cO_E})(\pi_f,-\epsilon)).$$ These homomorphisms
  are unique up to an element in $\cO_E(U_i)^\times.$ 
  
  For a given level $K_f$ we can find   a finite Galois extension $E/\Q$  such that all $\pi_f$ which occur in $H^\pkt(\SGK, \M_{\lambda,E})$  
 are weakly split  or what amounts to the same absolutely irreducible. We have the action of the Galois group
 $\Gal(E/\Q)$ on the set of isomorphism classes $\Coh(G, \lambda,K_f).$
 
  For a given $\pi_f$  we choose a covering $\Spec(\cO_{E(\pi_f)})$   by open subsets $U_1,U_2$ and generators  $T_i^{\rm alg}(\pi_f,\epsilon)(U_i).$  Let us call
 such a choice a {\it local trivialization} of ${\cal T}^{\rm alg}(\pi_f,\epsilon)$. Then it is clear that we choose our 
 trivialization such that it is invariant under the action of the Galois group: To get this  we choose  a $\pi_f$ in an orbit 
 and our local trivialization $\{T_i^{\rm alg}(\pi_f,\epsilon)\}_{i=1,2}$ over $\cO_{E(\pi_f)}$ as above. For 
 $\tau\in \Gal(E/\Q)$
  we 
 define 
 $$  
 T_i^{\rm alg}(\tau(\pi_f),\epsilon) =\tau(T_i^{\rm alg}(\pi_f,\epsilon))
  $$ 

  and then this system of local trivializations 
$$  \{  T_i^{\rm alg}(\tau(\pi_f),\epsilon) \}_{\pi_f\in \Coh_{!!}(G,\lambda,K_f), \epsilon,i }
  $$
  is defined over $\Q,$ i.e. invariant under the Galois group $\Gal(E/\Q).$

We return to the transcendental level. We assume $\pi_f\in \Coh_{!!}(G,\lambda,K_f)$ we choose a model space for
$\pi_f$   say $H_{\pi_f}=  H_{!!}^{b_n}(\SGK, \M_{\lambda,\cO_E})(\pi_f,+).$ For any $\iota: E\to \C$  we get
an inclusion 
\begin{align}
\Phi(\lambda,\pi_f,\iota):  \D_\lambda \otimes H_{\pi_f}\otimes_{E,\iota}\C \into \cA(G(\Q)\backslash G(\A))
\end{align}
which is Hecke equivariant   and therefore unique up to a scalar.
  This map provides an isomorphism for the cohomology 
   $$
    \Lambda^\bullet(\Phi(\lambda,\pi,\iota)): H^\bullet(  \fg_\Z/\fk_\Z , \D_{ \lambda}\otimes \M_{ \lambda})\otimes H_{\pi_f}\otimes_{E,\iota} \C  \iso H^\bullet(\SGK, \M_{\lambda,\C})(\pi_f)
     $$
which respects the action of $\pi_0(\Gl_n(\R)).$ We have seen in section  \ref{sec:HGK} that we have 
canonical generators for the $+$ and $-$ eigenspaces (See (\ref{genpm}))
\begin{align*}
 H^{b_n}(  \fg_\Z/\fk_\Z , \D_{ \lambda}\otimes \M_{ \lambda,\C}) =\C\omega_\lambda^{(+)}\oplus \C\omega_\lambda^{(-)}.
\end{align*}
( The choice of the generators was motivated by integrality considerations and in this sense they are canonical. 
But using the explicit description  we could just write them down in an ad hoc manner.   The actual choice
is not so important, what really matters is that they are "entangled", this means once we choose $\omega_\lambda^{(+)}$
the choice of $ \omega_\lambda^{(-)}$ is forced upon us. See \cite{ha-ra}, 5.2 ) 

For $\epsilon=\pm$ we have the two isomorphisms
\begin{align*}
H_{\pi_f}\otimes_{E,\iota}\C  \ppfeil{\Psi(\lambda,\pi_f,\epsilon,\iota)} H^{b_n}(\SGK,\M_{\lambda,E})(\pi_f,\epsilon)\otimes_{E,\iota}\C
\end{align*}
which are given by the composition $\psi_f\mapsto \omega_\lambda^{(\epsilon)}\times \psi_f\mapsto \Lambda^{b_n}(\Phi(\lambda,\pi_f,\iota)( \omega_\lambda^{(\epsilon)}\times \psi_f).$
We get a composition $T^{\text{trans}}(\pi_f,\iota,\epsilon)=\Psi(\lambda,\pi_f,\epsilon,\iota)^{-1}\circ \Psi(\lambda,\pi_f,-\epsilon,\iota)$
which is an isomorphism
\begin{align*}
T^{\text{trans}}(\pi_f,\iota,\epsilon):  H^{b_n}(\SGK,\M_{\lambda,E})(\pi_f,\epsilon)\otimes_{E,\iota}\C\to 
 H^{b_n}(\SGK,\M_{\lambda,E})(\pi_f,-\epsilon)\otimes_{E,\iota}\C
\end{align*}
and this isomorphism does not depend on the choice of the embedding   $\Phi(\lambda,\pi_f,\iota).$  For
$i=1,2$ we 
define the periods by comparing the two isomorphism between the $\pm$ eigenspaces (See also \cite{ha-ra}, 5.2.3):
\begin{align}
\Omega_i(\pi_f,\iota,\epsilon )T^{\text{trans}}(\pi_f,\iota,\epsilon)=   T_i^{\rm alg}(\pi_f,\epsilon)\otimes_{E,\iota} 1
\end{align}
The periods $\Omega_i(\pi_f,\iota,\epsilon ) $ are complex numbers which are well defined modulo
$\iota(\cO_{ (E(\pi_f)} (U_i))^\times$ and the ratio $\Omega_1(\pi_f,\iota,\epsilon ) /\Omega_2(\pi_f,\iota,\epsilon ) $
is an element in   $\iota(\cO_{E(\pi_f)})(U_1\cap U_2)^\times.$
 
 If we now work with this refined definition of the periods  the assertion in Theorem 7.39 in \cite{ha-ra}
 $$
\frac{1}{\Omega^{\varepsilon'}({}^\iota\!\sigma_f)} \, 
\frac{L^{\rm coh}\left(\iota, \sigma \times \sigma'^{\vv}, {\sf m}_0\right)}
{L^{\rm coh}\left(\iota , \sigma \times \sigma'^{\vv}, 1+{\sf m}_0\right)}
\ \in  \ 
\iota(E). 
$$
remains unchanged but now it makes sense to ask for the decomposition of these numbers into 
prime ideals. We have evidence that this  decomposition  into prime factors 
has some influence on the structure of cohomology of arithmetic groups.  The prime factors 
should be related to denominators of Eisensteinclasses but this relationship could be spoiled
by primes dividing the factor $c_\infty(\tm)$ above. (See also the above reference to SecOps.pdf
in \cite{harder-book}).

\end{document}